%% file: ostuni_2025.tex
\documentclass{daj}
\PassOptionsToPackage{numbers, compress}{natbib}
\usepackage{amsmath,amssymb,amsthm,fullpage,mathrsfs,pgf,tikz,caption,subcaption,mathtools}
\usepackage{natbib}
\usepackage[utf8]{inputenc} 
\usepackage[T1]{fontenc}    
\usepackage{hyperref}       
\usepackage{url}            
\usepackage{booktabs}       
\usepackage{amsfonts}       
\usepackage{nicefrac}       
\usepackage{microtype}      
\usepackage{times}
\usepackage{cleveref}
\usepackage[margin=1in]{geometry}
\usepackage{bbm}
\usepackage{thm-restate}
\usepackage{graphicx}
\usepackage{tikz}
\usetikzlibrary{shadings, calc}
\usepackage{comment}
\input{macro}
\allowdisplaybreaks

\dajAUTHORdetails{%
  title = {Strong Bounds for Skew-Corner-Free Sets}, 
  author = {Michael Jaber, Shachar Lovett, and Anthony Ostuni},
  plaintextauthor = {Michael Jaber, Shachar Lovett, and Anthony Ostuni},
  keywords = {skew corners, corners, arithmetic progressions},
}   

\dajEDITORdetails{%
   year={2025},
   number={20},
   received={16 April 2024},   
   published={30 September 2025},  
   doi={10.19086/da.144040},       
}   

\begin{document}

\begin{frontmatter}[classification=text]

\title{Strong Bounds for Skew-Corner-Free Sets} 

\author[michael]{Michael Jaber\thanks{Supported by NSF Grant CCF-2312573 and a Simons Investigator Award (\#409864, David Zuckerman).}}
\author[shachar]{Shachar Lovett\thanks{Supported by NSF DMS award 1953928, NSF CCF award 2006443, and a Simons investigator award. The research was done while on sabbatical at the Weizmann Institute of Science, partially funded by the European Research Council (ERC) under the European Union’s Horizon 2020 research and innovation programme, grant agreement No. 949083.}}
\author[anthony]{Anthony Ostuni\thanks{Supported by NSF DMS award 1953928.}}

\begin{abstract}
Motivated by applications to matrix multiplication algorithms, Pratt asked (ITCS'24) how large a subset of $[n] \times [n]$ could be without containing a \textit{skew corner}: three points $(x,y), (x,y+h),(x+h,y')$ with $h \ne 0$. We prove any skew-corner-free set has size at most $\exp(-\Omega(\log^{1/12} n))\cdot n^2$, nearly matching the best known lower bound of $\exp(-O(\sqrt{\log n}))\cdot n^2$ by Beker (arXiv'24). Our techniques generalize those of Kelley and Meka's recent breakthrough on three-term arithmetic progression (FOCS'23), answering a question of Beker (arXiv'24). We note that a similar bound was obtained concurrently and independently by Mili{\'c}evi{\'c} (arXiv'24).
\end{abstract}
\end{frontmatter}

\section{Introduction}
In an attempt to rule out certain groups as candidates for achieving an optimal matrix multiplication algorithm using the framework of \cite{cohn2003group, cohn2005group}, Pratt \cite{pratt2024generalized} introduced \textit{skew corners}: triples of points $(x,y), (x,y+h)$, and $(x+h,y')$ for arbitrary $x,y,y',$ and $h$. Note that the more traditional form of corners can be recovered by insisting $y=y'$. We call a skew corner trivial if $h=0$ and nontrivial otherwise, and we call a set skew-corner free if it does not contain any nontrivial skew corners. Pratt's matrix multiplication results are conditional, holding if the largest skew-corner-free subset of $[n] \times [n]$ has size $O(n^{1+\eps})$ for all $\eps>0$, as well as weaker results for the upper bound $O(n^{4/3-\eps})$ for any fixed $\eps > 0$ \cite{pratt2024generalized}.

Very recently, these two conjectures were refuted. Pohoata and Zakharov \cite{pohoata2024skew} constructed a skew-corner-free set of size $\Omega(n^{5/4})$, which was improved shortly after to $\exp(-O(\sqrt{\log n}))\cdot n^2$ by Beker \cite{beker2024improved} using a Behrend-type construction (see \cite{behrend1946sets}). In the same paper, Beker also improved the best known upper bounds to $O((\log n)^{-c})\cdot n^2$ for an absolute constant $c > 0$, and went on to ask if one could use the techniques of the recent breakthrough on three-term arithmetic progressions (3APs) \cite{kelley2023strong} (see also \cite{bloom2023kelley, bloom2023improvement}) to obtain a bound of $\exp(-\Omega(\log^c n))\cdot n^2$ for some constant $c > 0$. We resolve this question affirmatively.
\begin{restatable}{theorem}{mainint}\label{thm:corners_bounds_int}
    Any skew-corner-free subset of $[n] \times [n]$ has size at most $\exp\left\{-\Omega\left(\log^{1/12}n\right)\right\}\cdot n^2$.
\end{restatable}
In fact, we can prove such a result for any skew-corner-free subset of $G \times G$, where $G$ is an arbitrary finite Abelian group. It is likely that the constant $1/12$ can be somewhat improved using ideas similar to \cite{bloom2023improvement}. However, the overall technique's limitations appear to prevent improvements significant enough to approach Beker's lower bound \cite{beker2024improved}, so we instead prioritize a streamlined presentation. 

As is somewhat standard, we first exhibit our argument in the simpler finite field setting. Here, we can obtain slightly stronger bounds. 
\begin{restatable}{theorem}{mainff}\label{thm:corners_bounds_ff}
    Let $q$ be a prime power. Any skew-corner-free subset of $\F_q^n \times \F_q^n$ has size at most $q^{2n-\Omega(n^{1/9})}$.
\end{restatable}
A version of \Cref{thm:corners_bounds_ff} with the superior bound $O(q^{(2-c)n})$ (for a constant $c > 0$ depending on $q$) is known via the polynomial method \cite{pohoata2024skew}. However, this technique does not appear to generalize to the integer case, whereas our analytic approach is able to do so relatively painlessly.

We largely view this work as a step toward better understanding the traditional corners problem of how large a subset of $[n] \times [n]$ one can have without three points $(x,y), (x,y+h), (x+h, y)$. This problem is much older and better studied than the skew variant, yet the best upper bound of $O(n^2 / (\log \log n)^{0.0137\cdots})$ \cite{shkredov2006generalization} remains far from the lower bound of $\exp(-O(\sqrt{\log n}))\cdot n^2$ \cite{behrend1946sets} (see also \cite{linial2021larger, green2021lower, hunter2022corner}).

There has been optimism (see, e.g., \cite{meka2023talk} and \cite[Section 1.2]{peluse2023finite}) that the techniques recently used to dramatically improve the best known upper bounds on 3AP-free sets \cite{kelley2023strong} could also be applied in the (traditional) corners setting. Our work provides evidence that this is a potentially fruitful direction. In fact, we are able to either black-box or closely follow the proofs of many lemmas from the excellent exposition of Bloom and Sisask \cite{bloom2023kelley} of Kelley and Meka's work, raising the following question about the relationship between these problems.
\begin{question}
    Suppose the largest 3AP-free subset of $[n]$ has size $\delta n$. Does the largest skew-corner-free subset of $[n]\times [n]$ have size $\text{poly}(\delta)\cdot n^2$? More generally, is it possible to directly convert bounds from one setting to another?
\end{question}

\paragraph*{Concurrent work.} A similar result to \Cref{thm:corners_bounds_int} was obtained concurrently and independently by Mili{\'c}evi{\'c} \cite{milicevic2024good}. His techniques are also based on the ideas of \cite{kelley2023strong}, although the specifics diverge slightly from our own. 

\section{Background}\label{sec:background}

Let $G$ be a finite Abelian group throughout (although we will only truly be concerned with vector spaces over a finite field $\F_q^n$ or the integers modulo a positive integer $\mathbb{Z}/N \mathbb{Z}$).

\paragraph{Asymptotics.}
We use standard asymptotic notation of $O(\cdot)$ and $\Omega(\cdot)$ to suppress fixed constants that do not depend on any parameters. Occasionally, we will use subscripts to hide a dependence on particular parameters (e.g., $O_{\eps}(\cdot)$).

\paragraph{Distributions.}
A distribution $\nu$ over $G$ is a non-negative function $\nu\colon G \to \Rplus$ with $\E[\nu]=1$. For $f\colon G \to \R_{\ge 0}$ define $\nu(f)=\E_{x \in G} \nu(x) f(x)$ to be the average of $f$ under $\nu$. We also define $\nu_f$ to be $f$ normalized to have expectation $1$ under $\nu$; that is, $\nu_f(x) = \nu(x) f(x) / \nu(f)$. We use $\mu$ to denote the uniform distribution on $G$, so $\mu \equiv 1$. Given a set $A \subseteq G$, we identify $A$ with its indicator function $1_A$ and define accordingly $\nu(A)=\nu(1_A)$ and $\nu_A = \nu_{1_A}$. In this case, $\mu_A = \frac{|G|}{|A|}\cdot 1_A$, and in particular $\E_{x\in G}\mu_A(x) = 1$. In other words, $\mu_A$ is the normalized indicator function of $A$.

\paragraph{Functions.}

For functions $f,g \colon G \to \mathbb{R}$, we define inner products and norms with the normalized counting measure on $G$, namely
$$
    \langle f, g \rangle = \mathop{\mathbb{E}}_{x \in G} f(x) g(x) \quad\text{and}\quad \|f\|_p = \left( \mathop{\mathbb{E}}_{x \in G} |f(x)|^p \right)^{1/p} \text{ for } 1 \leq p < \infty,
$$
as well as $\|f\|_{\infty} = \max_{x\in G} |f(x)|$. We will also want to work with other measures on $G$. For a measure $\nu$ on $G$, we write
$$
    \langle f, g \rangle_{\nu} = \mathop{\E}_{x \in G} \nu(x) f(x) g(x) \quad\text{and}\quad
    \|f\|_{p(\nu)} = \left( \mathop{\E}_{x \in G} \nu(x) |f(x)|^p \right)^{1/p} \text{    for    }1 \leq p < \infty.
$$
Additionally, we define the convolution and the difference convolution as 
$$
    (f \ast g)(x) = \mathop{\mathbb{E}}_{y \in G} f(y) g(x-y) \quad\text{and}\quad
    (f \circ g)(x) = \mathop{\mathbb{E}}_{y \in G} f(y) g(x+y),
$$
and the $p$-fold convolution as $f^{(p)} = f\ast f^{(p-1)}$, where $f^{(1)}=f$. Note the useful adjoint property $\langle f, g \ast h \rangle = \langle h \circ f, g \rangle.$ Occasionally, the notation $f^g(x)=f(x-g)$ for some $g\in G$ will be useful.

Let $\widehat{G} = \{\text{homomorphisms }\gamma : G \to \mathbb{C}^\times\}$ be the dual group of $G$. The Fourier transform of $f$ is $\widehat{f} \colon \widehat{G} \to \mathbb{C}$, where
\[
\widehat{f}(\gamma) = \mathop{\E}_{x\in G} f(x)\gamma(-x).
\]
If $\widehat{f}(\gamma) \ge 0$ for all $\gamma \in \widehat{G}$, we say $f$ is \textit{spectrally non-negative}, denoted $\widehat{f} \ge 0$. Note the basic facts that $\widehat{f \ast g} = \widehat{f}\cdot\widehat{g}$ and $\widehat{f \circ f} = |\widehat{f}|^2$. In particular, $f \circ f$ is spectrally non-negative. We define the convolution for two functions $f, g \colon \widehat{G} \to \mathbb{C}$ using the ordinary counting measure:
\[
    (f \ast g)(\gamma) = \sum_{\tau \in \widehat{G}} f(\tau)g(\gamma - \tau).
\]
The use of the ordinary counting measure here is so that we have the identity $\widehat{f^p} = \widehat{f}^{(p)}$.

\subsection{Skew corners}
Sequences of evenly spaced integers, known as \textit{arithmetic progressions}, are central objects of study in additive combinatorics. In an attempt to understand similar structure in higher dimensions, Ajtai and Szemer{\'e}di \cite{ajtai1974sets} considered the following objects.
\begin{definition}[Corners]
    A \emph{corner} in $G \times G$ is a triple of the form $(x,y), (x, y+h), (x+h, y)$ for arbitrary $x,y,h \in G$. A corner is \emph{nontrivial} if $h \ne 0$.
\end{definition}

By relaxing the structure to allow the third point to lie anywhere on the vertical line $x+h$, one obtains \textit{skew corners}, a generalization recently raised by Pratt \cite{pratt2024generalized}.
\begin{definition}[Skew corners]
    A \emph{skew corner} in $G \times G$ is a triple of the form $(x,y), (x, y+h), (x+h, y')$ for arbitrary $x,y,y',h \in G$. A skew corner is \emph{nontrivial} if $h \ne 0$. Additionally, a set is \emph{skew-corner-free} is it does not contain nontrivial skew corners.
\end{definition}

One immediate, yet useful, observation is that shifting a set will not change the number of skew corners it contains.
\begin{observation}\label{obs:shift}
    The number of skew corners in a set $A$ is invariant to both horizontal shifting ($A \rightarrow \{(x+a,y): (x,y) \in A\}$) and vertical shifting ($A \rightarrow \{(x,y+b(x)) : (x,y) \in A\}$).
\end{observation}

The following lemma allows us to count skew corners in an analytic way. 
\begin{lemma}\label{lem:corners_count}
    Let $A \subseteq G \times G$ and $A_x = \{y : (x,y) \in A\}$. Additionally, let $D \colon G \to \Rplus$ be given by $D(x) = |A_x|$. Then, the number of skew corners in $A$ is counted by 
    $$
        |G|^2 \sum_{x\in G} \langle 1_{A_x} \circ 1_{A_x+x}, D \rangle.
    $$
\end{lemma}
\begin{proof}
    Recall a skew corner in $A$ is a triple of the form 
    $$
        (x,y), (x,y+h), (x+h, y')
    $$
    for arbitrary $x,y,y',h\in G$. The number of skew corners is given by 
    \begin{align*}
        \sum_{x,y,y',h \in G} 1_A(x,y) 1_A(x,y+h) 1_A(x+h,y') &= \sum_{x,y,y',h} 1_{A_x}(y) 1_{A_x}(y+h) 1_{A_{x+h}}(y') \\
        &= \sum_{x,y,h} 1_{A_x}(y) 1_{A_x+x}(x+y+h) D(x+h) \\
        &= |G| \sum_{x,h}  (1_{A_x} \circ 1_{A_x+x}) (x+h) D(x+h) \\
        &= |G|^2 \sum_{x} \langle 1_{A_x} \circ 1_{A_x+x}, D \rangle. \tag*{\qedhere}
    \end{align*}
\end{proof}

Occasionally, the following normalized form will be more useful.
\begin{lemma}\label{lem:normalized_corners_count}
    Let $A \subseteq G \times G$, $A_x = \{y : (x,y) \in A\}$, and $\eta \in \Rplus$. Additionally, let $f = \sum_{x \in G} 1_{A_x} \circ 1_{A_x+x}$ and $D \colon G \to \Rplus$ be given by $D(x) = |A_x|$. If $\langle \mu_f, \mu_D\rangle \geq \eta$, then the number of skew corners in $A$ is at least 
    \[
        \eta\cdot\frac{|A|^3}{|G|^2}.
    \]
\end{lemma}
\begin{proof}
    Unraveling the definition of $\langle \mu_f, \mu_D\rangle$, we have
    \begin{align*}
       \eta \le \left\langle \frac{1}{\E[f]}  \sum_{x \in G} 1_{A_x} \circ 1_{A_x+x}, \frac{D}{\E[D]}\right\rangle &= \frac{1}{\E[f]}\cdot\frac{1}{\E[D]} \sum_{x \in G}\left\langle 1_{A_x} \circ 1_{A_x+x}, D\right\rangle \\
        &= \frac{|G|^2}{\sum |A_x|^2}\cdot\frac{|G|} {\sum |A_x|} \sum_{x \in G}\left\langle 1_{A_x} \circ 1_{A_x+x}, D\right\rangle \\
        &\le \frac{|G|^3}{|A|^2}\cdot\frac{|G|}{|A|} \sum_{x \in G}\left\langle 1_{A_x} \circ 1_{A_x+x}, D\right\rangle \tag{Cauchy-Schwarz} \\
        &= \frac{|G|^2}{|A|^3} \times \text{(number of skew corners in $A$)}. \tag*{(\Cref{lem:corners_count}) \quad \qedhere}
    \end{align*}
\end{proof}

\subsection{Bohr sets}\label{sec:bohr}
To prove \Cref{thm:corners_bounds_ff}, we will take advantage of the abundance of subspaces in $\F_q^n$. As this no longer holds once we move to more general Abelian groups, we will turn to the machinery of \textit{Bohr sets}, introduced by Bourgain \cite{bourgain1999triples}, to prove \Cref{thm:corners_bounds_int}. We will only require the very basics, but one may entirely skip over this topic if they are solely interested in a high-level understanding of our results; all the main ideas are present in the simpler finite field case (\Cref{sec:finite_field}). 

\begin{definition}[Bohr sets]
    For a nonempty $\Gamma \subseteq \widehat{G}$ and $\phi \in [0,2]$, we define the Bohr set
    \[
        B = \mathrm{Bohr}(\Gamma, \phi) = \{x \in G : |1-\gamma(x)| \le \phi \text{ for all } \gamma \in \Gamma\},
    \]
    where $\Gamma$ is the \textit{frequency set}, $\phi$ is the \textit{width} (or \textit{radius}), and $|\Gamma|$ is the \textit{rank} $\rk(B)$.
\end{definition}

It follows from the definition that Bohr sets are symmetric, meaning that if $x \in B$, we must also have $-x \in B$. Note also that a Bohr set $B$ does not uniquely define its frequency set or width. Thus, we always implicitly view $B$ as the tuple $(\Gamma, \phi, \mathrm{Bohr}(\Gamma, \phi))$. However, the notation $B' \subseteq B$ only refers to containment of the actual set itself and does not depend on the frequency set or width. Indeed, there are Bohr sets $B, B'$ where $B' \subset B$, but the frequency set of $B$ is not contained in the frequency set of $B'$, and the width of $B'$ is not smaller than than the width of $B$. To handle these situations, we follow \cite{schoen2016roth} in defining \textit{sub-Bohr} sets.
\begin{definition}
    Let $B=\mathrm{Bohr}(\Gamma, \phi)$, $B'=\mathrm{Bohr}(\Gamma', \phi')$ be Bohr sets.
    We say that $B'$ is a \emph{sub-Bohr set} of $B$, denoted $B' \le B$, if $\Gamma' \supseteq \Gamma$ and $\phi' \leq \phi$.
\end{definition}

\begin{observation}
    Let $B' \leq B$ for Bohr sets $B, B'$. Then, $B' \subseteq B$.
\end{observation}

A common operation performed on Bohr sets is \textit{dilation}. If $B=\mathrm{Bohr}(\Gamma,\phi)$ is a Bohr set and $\rho>0$, then the corresponding \textit{dilate} of $B$ is $B_{\rho}=\mathrm{Bohr}(\Gamma, \rho \phi)$. The following observation about sub-Bohr set dilates will be useful in later arguments.

\begin{observation}\label{obs:subbohrdil}
    Let $B' \leq B$ for Bohr sets $B, B'$. Then, $B'_{\rho} \leq B_{\rho}$ for all $\rho > 0$. In particular, $B'_{\rho} \subseteq B_{\rho}$.
\end{observation}

If small dilations do not substantially change the size of a Bohr set $B$, we call $B$ \textit{regular}. Regular Bohr sets will be particularly useful for our purposes, as informally, one can view a regular Bohr set of rank $r$ analogously to a subspace of co-dimension $r$.

\begin{definition}[Regular Bohr sets]
    A Bohr set $B$ of rank $r$ is \emph{regular} if for all $0 \le \kappa \le 1/100r$, we have 
    \[
    |B_{1+\kappa}| \le (1+100\kappa r)|B| \quad\text{and}\quad |B_{1-\kappa}| \ge (1-100\kappa r)|B|.
    \]
\end{definition}

We record some useful lemmas, most of which are listed in the appendix of \cite{bloom2023kelley}. See \cite[Section 4.4]{tao2006additive} or \cite[Chapter 6]{bloom21} for proofs of the first two standard results, as well as a more thorough treatment of Bohr set machinery.

\begin{lemma}\label{lem:regular_dilation}
    For any Bohr set $B$, there exists $\rho \in [\frac{1}{2}, 1]$ such that $B_{\rho}$ is regular. 
\end{lemma}

\begin{lemma}\label{lem:bohrsize}
    If $\rho \in (0,1)$ and $B$ is a Bohr set of rank $r$, then $|B_{\rho}| \geq (\rho/4)^r |B|$.
\end{lemma}

\begin{lemma}[{\cite[Lemma 4.6]{bloom2020breaking}}]\label{lem:boundingWithDilation}
    There is a constant $c > 0$ such that the following holds. If $B$ is a regular Bohr set of rank $r$ and $\nu$ is a probability measure supported on $B_{\rho}$, with $\rho \le c/r$, then
    $$
        \mu_{B} \leq 2\mu_{B_{1 + \rho}} \ast \nu. 
    $$
\end{lemma}

\begin{lemma}[{\cite[Lemma 4.5]{bloom2020breaking}}]\label{lem:BohrConv}
    If $B$ is a regular Bohr set of rank $r$ and $\nu$ is a probability measure supported on $B_{\rho}$, with $\rho \in (0,1)$, then
    $$
        \|\mu_{B} \ast \nu - \mu_B \|_1 \le  O(\rho r).
    $$
\end{lemma}

We will make use of the following corollary. 

\begin{corollary}\label{cor:approxbohr}
    If $B$ is a regular Bohr set of rank $r$ and $\nu$ is a probability measure supported on $B_{\rho}$, with $\rho \in (0,1)$, and $f \colon G \to \mathbb{R}$, then  
    $$
            |\langle f \ast \nu, \mu_B\rangle - \langle f, \mu_B \rangle| \le O\left(\|f\|_{\infty} \cdot \rho r\right). 
    $$
\end{corollary}

\begin{proof}
    We have $\langle f \ast \nu, \mu_B \rangle = \langle f, \mu_B \circ \nu \rangle = \langle f, \mu_B\rangle + \langle f, \mu_B \circ \nu - \mu_B\rangle.$
    Since Bohr sets are symmetric, we can replace $\nu(x)$ with $\nu(-x)$ and maintain that $\nu$ is supported on $B_{\rho}$. It suffices to bound
    \begin{align*}
        |\langle f, \mu_B \ast \nu - \mu_B \rangle| \le \|f \|_{\infty} \| \mu_B \ast \nu - \mu_B \|_{1} \le O(\|f \|_{\infty} \cdot \rho r),
    \end{align*}
    which follows by \Cref{lem:BohrConv} and the triangle inequality. 
\end{proof}
We briefly remark that in several subsequent arguments, we will again exploit the symmetry of subspaces and Bohr sets $V$ to use the identity $\mu_V(x) = \mu_V(-x)$ for all $x$.

\subsection{Spread sets}

A key definition underpinning Kelley and Meka's breakthrough result on 3APs \cite{kelley2023strong} is a pseudorandomness notion they call \textit{spreadness}. A subset of $\F_q^n$ is spread if it has no significant density increment by restricting to a large affine subspace.
\begin{definition}[$(r, \lambda)$-spread]\label{def:og_spread}
    Let $r \ge 1$ and $\lambda > 1$. 
    A set $A \subseteq \F_q^n$ is \emph{$(r, \lambda)$-spread} if for every subspace $V \subseteq \F_q^n$ of co-dimension at most $r$ and shift $x \in \F_q^n$, we have
    \[
        \frac{|(A-x) \cap V|}{|V|} \le \lambda \cdot \frac{|A|}{|\F_q^n|}.
    \]
\end{definition} 

In light of \Cref{lem:corners_count}, we will need the following generalization of spreadness that applies to multiple sets. It should intuitively be clear that the definition we choose needs to be invariant to shifts, since the number of skew corners in a set has this property (\Cref{obs:shift}).
\begin{definition}[$(r, \lambda)$-simultaneously spread]\label{def:sim_spread_ff}
    Let $r \ge 1$ and $\lambda > 1$.
    A collection $\{A_i\}_{i \in S}$ of subsets of $\F_q^n$ is \emph{$(r, \lambda)$-simultaneously spread} if for every subspace $V \subseteq \F_q^n$ of co-dimension at most $r$ and shifts $\{x_i\}_{i \in S}$ in $\F_q^n$, we have
    \[
        \sum_{i \in S} \left(\frac{|(A_i-x_i) \cap V|}{|V|}\right)^2 \leq \lambda \cdot \sum_{i \in S}\left(\frac{|A_i|}{|\F_q^n|} \right)^2.
    \]
\end{definition}
When our ambient space is a vector space $W \subseteq \F_q^n$, we will say the collection is $(r,\lambda)$-simultaneously spread \emph{in $W$} if it is $(r,\lambda)$-simultaneously spread after being mapped according to an isomorphism from $W$ to $\F_q^{\dim W}$.

Additionally, we require a version with Bohr sets replacing subspaces for general Abelian groups. For technical reasons, it will simplify later analysis to measure density with respect to a particular Bohr set rather than the ambient space (as in \Cref{def:sim_spread_ff}).

\begin{definition}[$(r, \delta, \lambda)$-simultaneously spread]\label{def:sim_spread_int}
    Let $r \ge 1$, $\delta\in (0,1)$, and $\lambda > 1$. 
    A collection $\{A_i\}_{i \in S}$ of subsets of a regular Bohr set $B\subseteq G$ is \emph{$(r, \delta, \lambda)$-simultaneously spread in $B$} if for every regular Bohr set $B' \leq B$ of rank at most $\rk(B)+r$ and measure 
    $\mu(B') \geq \delta \mu(B)$, and shifts $\{x_i\}_{i \in S}$ in $G$, we have
    \[
        \sum_{i \in S} \left(\frac{|(A_i-x_i) \cap B'|}{|B'|}\right)^2 \leq \lambda \cdot \sum_{i \in S}\left( \frac{|A_i|}{|B|} \right)^2.
    \]
\end{definition}

It was useful in \cite{kelley2023strong} to use $\|\mu_A \ast \mu_V\|_{\infty} \le \lambda$ as an equivalent form of \Cref{def:og_spread}. We will later need to exploit a similar relationship with an alternative definition of simultaneous spreadness. (The proof of this connection explains why the summands in \Cref{def:sim_spread_ff} and \Cref{def:sim_spread_int} are squared.)

\begin{lemma}\label{lem:inftospread_ff}
    Let $r \ge 1$, $\lambda > 1$, $\{A_i\}_{i \in S}$ be an $(r, \lambda)$-simultaneously spread collection of subsets of $\F_q^n$, and $\alpha = \sum_{i \in S} (|A_i|/|\F_q^n|)^2$. Then for every subspace $V \subseteq \F_q^n$ of co-dimension at most $r$, we have
    \[
        \left \|\left(\frac{1}{\alpha}\sum_{i \in S} 1_{A_i} \circ 1_{A_i} \right)  \ast \mu_{V} \right\|_{\infty} \le \sqrt{\lambda}. 
    \]
\end{lemma}
\begin{proof}
    Let $x\in \F_q^n$ be an element maximizing the input  to the $L^{\infty}$-norm. Then by Cauchy-Schwarz,
    \[
        \left \|\left(\sum_{i \in S} 1_{A_i} \circ 1_{A_i} \right)  \ast \mu_{V} \right\|_{\infty} = \sum_{i \in S} \langle 1_{A_i} \circ 1_{A_i}, \mu_{V+x} \rangle \leq \sum_{i \in S} \|1_{A_i}\|_{1} \|1_{A_i} \ast \mu_{V}\|_{\infty} \le \alpha^{1/2} \left( \sum_{i \in S} \| 1_{A_i} \ast \mu_{V} \|_{\infty}^2 \right)^{1/2}.
    \]
    Now let $x_i$ be a shift which maximizes $|(A_i - x_i) \cap V|$ for each $i \in S$. By our spreadness assumption, we have
    \[
        \sum_{i \in S} \| 1_{A_i} \ast \mu_{V} \|_{\infty}^2 = \sum_{i \in S} \left( \frac{|(A_i - x_i) \cap V|}{|V|}\right)^2 \le \lambda \cdot \sum_{i \in S}\left(\frac{|A_i|}{|\F_q^n|} \right)^2 = \lambda\alpha.
    \]
    Substituting the second equation into the first yields the result.
\end{proof}

We also need an analogous lemma for Bohr sets. While similar, we present the proof here to illustrate why there is an additional $\mu(B)^{-1}$ factor.
\begin{lemma}\label{lem:inftospread}
    Let $r \ge 1$, $\eps \in (0,1)$, $\lambda > 1$, $\{A_i\}_{i \in S}$ be an $(r, \delta, \lambda)$-simultaneously spread collection of subsets of a regular Bohr set $B \subseteq G$, and $\alpha = \sum_{i \in S} (|A_i|/|G|)^2$. Then for every regular Bohr set $B' \le B$ of rank at most $\rk(B)+r$ and measure $\mu(B') \ge \delta\mu(B)$, we have
    \[
        \left \|\left(\frac{1}{\alpha}\sum_{i \in S} 1_{A_i} \circ 1_{A_i} \right)  \ast \mu_{B'} \right\|_{\infty} \le \sqrt{\lambda}\cdot\mu(B)^{-1}. 
    \]
\end{lemma}
\begin{proof}
    Let $x\in G$ be an element maximizing the input to the $L^{\infty}$-norm. Then by Cauchy-Schwarz,
    \[
        \left \|\left(\sum_{i \in S} 1_{A_i} \circ 1_{A_i} \right)  \ast \mu_{B'} \right\|_{\infty} = \sum_{i \in S} \langle 1_{A_i} \circ 1_{A_i}, \mu_{B'+x} \rangle \leq \sum_{i \in S} \|1_{A_i}\|_{1} \|1_{A_i} \ast \mu_{B'}\|_{\infty} \le \alpha^{1/2} \left( \sum_{i \in S} \| 1_{A_i} \ast \mu_{B'} \|_{\infty}^2 \right)^{1/2}.
    \]
    Now let $x_i$ be a shift which maximizes $|(A_i - x_i) \cap B'|$. By our spreadness assumption, we have
    \[
        \sum_{i \in S} \| 1_{A_i} \ast \mu_{B'} \|_{\infty}^2 = \sum_{i \in S} \left( \frac{|(A_i - x_i) \cap B'|}{|B'|}\right)^2 \le \lambda \cdot \sum_{i \in S}\left(\frac{|A_i|}{|B|} \right)^2 = \lambda\alpha\cdot\mu(B)^{-2}.
    \]
    Substituting the second equation into the first yields the result.
\end{proof}

\section{The finite field case}\label{sec:finite_field}
In this section, we prove \Cref{thm:corners_bounds_ff}, restated below for convenience. Note that Pohoata and Zakharov \cite{pohoata2024skew} used the polynomial method to obtain the stronger bound $O(q^{(2-c)n})$ (for a constant $c > 0$ depending on $q$). However, their techniques do not seem to generalize to the integer case.
\mainff*

The analogous result for 3APs in \cite{kelley2023strong} is proved by repeatedly restricting the set of interest to a large, denser subspace until no such subspaces exist (i.e., the set is spread). Then, they show spread sets have approximately the same number of 3APs as a random set, and since there cannot be too many trivial 3APs, the original set size cannot be too large.

This intuition guides our approach, although the situation is more complicated for skew corners. It is helpful to think about the converse of the above discussion. For a set $A$ of large density, if $A$ contains no 3APs, then $A$ must exhibit a density increment on a reasonably large subspace. In the skew corners case, we start with a large subset $A$ of $\F_q^n \times \F_q^n$, and denote by $\{A_g\}$ its columns. We are asking if there are many solutions of the form $x - y + g \in D$ for $x,y \in A_g$, where $D$ can be thought of as the set of columns of decent density. By Kelley and Meka's result, if each of the sets $A_g$ was spread, then we would have many solutions and thus many skew corners. So, we know that if there are no skew corners, then most of the $A_g$ are \textit{not} spread, and most of the $A_g$ exhibit a density increment on some large subspace $V_g$. The issue is, however, that a priori these subspaces $V_g$ might be completely uncorrelated. 

One might hope that if the sum of $1_{A_g} \circ 1_{A_g+g}$ is far from uniform, then the individual deviations of the $1_{A_g} \circ 1_{A_g+g}$ are cooperating, and so there should be a density increment which is compatible with the $A_g$ on average. Indeed, this ends up being what happens, and we show that there is a common subspace $V$ which ``explains'' the lack of uniformity for the sum of convolutions. Namely, for most $A_g$, there is some affine shift $V+x_g$ which admits a density increment on $A_g$. For technical reasons, we will measure progress in terms of the sum of the squared densities $(|A_g \cap (V+x_g)|/|V|)^2$. 

It is tempting to ignore the collection of $A_g$, and simply proceed with a single $A_h$ of highest density. The Kelley-Meka techniques allow one to restrict $A_h$ to an affine subspace $V$ on which the convolution $1_{A_h} \circ 1_{A_h+h}$ is close to uniform. The issue, however, is that this approach does not ensure that $D$ has large density on $V$. There may in fact be no nonempty columns indexed by elements from $V$, and so this naive approach does not work as is. For this reason, our density increment process needs access to many of the $A_g$ at once, so that we can restrict the set of ``rows'' in which the $A_g$ are contained, as well as the set of ``columns'' by which the $A_g$ are indexed while maintaining density guarantees.

The main steps then are to argue that (1) a collection of sets $\{A_g\}_{g \in G}$ can be restricted in such a way that it no longer admits a collective density increment on a subspace and (2) if a collection of $\{A_g\}_{g \in G}$ contains enough points but fails to have skew corners, then there must be some subspace on which the $A_g$ collectively obtain a density increment.

\subsection{Obtaining simultaneous spreadness}
A useful property of spreadness is that it is easy to obtain via an iterative density increment argument. The generalization to simultaneous spreadness preserves this property with only slightly more work.

\begin{lemma}\label{lem:spreadness_ff}
    Let $\eps \in (0,1), r \ge 1,$ and $\{A_g\}_{g \in \F_q^n}$ be a collection of subsets of $\F_q^n$ with $\sum |A_g|^2 \geq 2^{-d}|\F_q^n|^3$. Then there exists a subspace $V \subseteq \F_q^n$ of co-dimension at most $O(rd/\eps)$ and shifts $\{x_g\}_{g \in \F_q^n}$, $w \in \F_q^n$ such that
    \begin{enumerate}
        \item The collection $\{(A_g- x_g) \cap V\}_{g \in V+w}$ is $(r, 1+\eps)$-simultaneously spread in $V$, and
        \item $\sum_{g \in V+w} |(A_g-x_g) \cap V|^2 \geq 2^{-d}|V|^3$.
    \end{enumerate}
\end{lemma}
\begin{proof}
    If $\{A_g\}_{g \in \F_q^n}$ is not $(r,1+\eps)$-simultaneously spread, then there exists a subspace $V \subseteq \F_q^n$ of co-dimension at most $r$ and shifts $\{x_g\}_{g \in \F_q^n}$ in $\F_q^n$ such that
    \[
        \sum_{g \in \F_q^n} \left(\frac{|(A_g-x_g) \cap V|}{|V|}\right)^2 > (1+\eps) \sum_{g \in \F_q^n}\left( \frac{|A_g|}{|\F_q^n|} \right)^2.
    \]
    By averaging, there must exist an affine subspace $V+w$ with
    \begin{equation}\label{eq:avg_subspace}
        \sum_{g \in V+w} \left(\frac{|(A_g-x_g) \cap V|}{|V|}\right)^2 > (1+\eps) \frac{|V|}{|\F_q^n|} \sum_{g \in \F_q^n}\left( \frac{|A_g|}{|\F_q^n|} \right)^2.
    \end{equation}
    Construct a new collection $A'_g = (A_{g+w} - x_{g+w}) \cap V$ for $g \in V$. If $\{A'_g\}_{g \in V}$ is not $(r,1+\eps)$-simultaneously spread in $V$, we repeat this process with $\{A'_g\}_{g \in V}$ as the collection and $V \cong \F_q^{\dim(V)}$ as the ambient vector space, noting that \Cref{eq:avg_subspace} and our initial density assumption imply
    \[
        \sum_{g\in V}|A'_g|^2 > (1+\eps) \cdot 2^{-d}|V|^3.
    \]
    After $i$ iterations, the density will be at least $(1+\eps)^i\cdot2^{-d}$. Since it cannot exceed 1, we must obtain the desired collection in $O(d/\eps)$ iterations.
\end{proof}

\subsection{Using simultaneous spreadness}
Now that we have a simultaneously spread collection, it remains to show the corresponding set must have roughly as many skew corners as a random set of the same density. This section will focus on proving the following theorem.

\begin{restatable}{theorem}{strVSpsdFF}\label{thm:strucvspseudo_ff}
    Let $\eps \in (0,1)$. For a collection $\{A_g \}_{g \in \F_q^n}$ of subsets of $\F_q^n$ with $\sum_{g \in \F_q^n} |A_g| \geq 2^{-{d}}|\F_q^n|^2$, let $F = \sum_{g \in \F_q^n} 1_{A_g} \circ 1_{A_g+g}$, $F' = \sum_{g \in \F_q^n} 1_{A_g} \circ 1_{A_g}$, and $D \colon \F_q^n \to \Rplus$ be given by $D(x) = |A_x|$. Then either
    \begin{enumerate}
        \item 
        $|\langle \mu_F, \mu_D \rangle - 1| \leq \eps$, or
        \item
        There exists a subspace $V$ of co-dimension
        $O_{\eps}(d^8)$
        such that $\|\mu_{F'} \ast \mu_V\|_{\infty} \geq 1 + \frac{1}{32}\eps$.
    \end{enumerate}
\end{restatable}

The proof of 
\Cref{thm:corners_bounds_ff} follows from \Cref{thm:strucvspseudo_ff}. Namely, if our collection of sets $\{A_g\}_{g \in G}$ admits no collective density increment, and it is sufficiently dense, then it must contain roughly as many skew corners as a random set. 

\begin{proof}[Proof of \Cref{thm:corners_bounds_ff} from \Cref{thm:strucvspseudo_ff}]
    Let $A \subseteq \F_q^n \times \F_q^n$ be skew-corner-free with $|A| = 2^{-d}|\F_q^n|^2$, and define the collection $\{A_g\}_{g\in\F_q^n}$ by $A_g = \{y : (g,y) \in A\}$. By Cauchy-Schwarz, we have $\sum|A_g|^2 \ge 2^{-2d}|\F_q^n|^3$, so applying \Cref{lem:spreadness_ff} with $\{A_g\}_{g\in\F_q^n}$, $\eps\coloneqq 1/128$, $d\coloneqq 2d$, and $r>0$ to be determined later yields a collection of sets $\{(A_{g+w}-x_{g+w}) \cap V\}_{g \in V}$, where $V \subseteq \F_q^n$ is a co-dimension $O(dr)$ subspace and $\{x_g\}_{g \in \F_q^n}, w \in \F_q^n$ are shifts. For brevity, define $A'_g \coloneqq (A_{g+w} - x_{g+w}) \cap V$ and $A' \coloneqq \{(g,y): g \in V, y \in A'_g\} \subseteq V \times V$. Note that by \Cref{obs:shift}, $A$ has at least as many skew corners as $A'$, $|A'| \ge 2^{-2d} |V|^2$, and $\{A'_g\}_{g \in V}$ is $(r,1+1/128)$-simultaneously spread in $V$. 

    Apply \Cref{thm:strucvspseudo_ff} with $\eps \coloneqq 1/2$ to the new collection, and set $r=O(d^8)$ to be larger than the co-dimension bound in the second conclusion. If $|\langle \mu_F, \mu_D \rangle - 1| \leq 1/2$, \Cref{lem:normalized_corners_count} implies $A'$ contains at least
    \[
    \frac{1}{2}\cdot\frac{|A'|^3}{|V|^2} \ge 2^{-6d-1}|V|^4 \ge 2^{-\Omega(d^9)}|\F_q^n|^4
    \]
    skew corners. Since there are at most $|\F_q^n|^3$ trivial skew corners in $A'$, we are done after rearrangement.
    Otherwise, there exists a subspace $V$ of co-dimension 
    $\ell = O(d^8)$, with $\ell < r$, such that $\|\mu_{F'} \ast \mu_{V}\|_{\infty} \geq 1 + 1/64$. By \Cref{lem:inftospread_ff}, this contradicts the fact that $\{A'_{g}\}_{g \in V}$ is $(r,1+1/128)$-simultaneously spread.
\end{proof}

The proof of \Cref{thm:strucvspseudo_ff} follows from five major steps, which we briefly sketch. Throughout, let $F = \sum_g 1_{A_g} \circ 1_{A_g+g}$, $F' = \sum_g 1_{A_g} \circ 1_{A_g}$, and $D(x) = |A_x|$.
\begin{enumerate}
    \item Non-uniformity (\Cref{subsec:non-uniformity}) -- If there are few skew corners (i.e., $|\langle \mu_F, \mu_D \rangle - 1| > \eps$), then H\"{o}lder's inequality implies that $\|\mu_F - 1\|_p \geq \eps / 2$ for $p = O(d)$.

    \item Unbalancing (\Cref{subsec:unbalancing}) -- We then show it suffices to consider the spectrally non-negative $F'$, which we use to prove $\|\mu_{F'}\|_{p'} \geq 1+\eps / 4$ for some $p'=O_{\eps}(p)$.

    \item Dependent random choice (\Cref{subsec:sifting}) -- From here, we rely on a known application of dependent random choice to find two relatively dense sets $M_1, M_2$ whose difference convolution correlates well with elements of noticeable upward deviation from $F'$'s mean. That is, $\langle \mu_{M_1} \circ \mu_{M_2}, 1_S \rangle \ge 1-\eps/32$, where $S = \{x : \mu_{F'}(x) > 1 + \eps/8\}$.

    \item Almost periodicity (\Cref{subsec:almostPeriod}) -- Next, we use an almost periodicity result of \cite{schoen2016roth} to show there exists a subspace $V$ of co-dimension $O_{\eps}(d^8)$ such that
    $\langle \mu_V \ast \mu_{M_1} \circ \mu_{M_2}, 1_S \rangle \geq 1 - \eps/16$.

    \item Combining steps (\Cref{subsec:combining}) -- Putting these results together yields the second conclusion of \Cref{thm:strucvspseudo_ff}.
\end{enumerate}
These are essentially the same steps given in \cite{bloom2023kelley}, which is an excellent exposition and simplification of \cite{kelley2023strong}. Moreover, we are able to either black-box or closely follow the proofs of many of their results, highlighting the relationship between arithmetic progressions and skew corners.

\subsubsection{Non-uniformity}\label{subsec:non-uniformity}
The following lemma will allow us to argue that if there are very few skew corners, then $\|\mu_F - 1\|_p \geq \eps/2$. 

\begin{lemma}\label{lem:nonuniform_ff}
    Let $\eps > 0$ and $f \colon G \to \R$. If $\{A_g\}_{g \in G}$ is a collection of subsets of $G$ where $\sum_{g \in G} |A_g| \geq 2^{-d}|G|^2$, and $D \colon G \to \Rplus$ is given by $D(x) = |A_x|$, then either 
    \begin{enumerate}
        \item $|\langle \mu_f, \mu_D \rangle - 1| \leq \eps$ or
        
        \item $\|\mu_f - 1\|_p \geq \eps / 2$ for an integer $p = O(d)$. 
    \end{enumerate}
\end{lemma}
\begin{proof}
    If (1) fails, then by H\"{o}lder's inequality, we have for any $p \ge 1$ that
    \[
        \eps < \left | \langle \mu_{f} - 1, \mu_D \rangle \right | \leq \| \mu_{f} - 1 \|_{p} \| \mu_D \|_{1 + \frac{1}{p-1}} \leq \| \mu_f - 1 \|_{p} \| \mu_D \|_{\infty}^{1/p}  \|\mu_D\|_{1}^{1 - 1/p} = \| \mu_f - 1\|_p \|\mu_D\|_{\infty}^{1/p}.
    \]
    Note that $\|\mu_D \|_{\infty} \leq 2^d$, since $D(x) \leq |G|$ and $\mathbb{E}[D] \geq 2^{-d}|G|$. Setting $p = \ceil{d}$ yields (2). 
\end{proof}

\subsubsection{Unbalancing}\label{subsec:unbalancing}
The next lemma will allow us to consider the function $F' = \sum_g 1_{A_g} \circ 1_{A_g}$ rather than $F = \sum_g 1_{A_g} \circ 1_{A_g+g}$. 

\begin{lemma}\label{lem:decoupling_ff}
    Let $p \geq 1$ be an integer and $\{f_g \}_{g \in S}$ a collection of real-valued functions on $G$ indexed by some set $S \subseteq G$, we have
    $$
        \left \| \sum_{g \in S} f_g \circ f_g \right \|_p \geq \left \| \sum_{g \in S} f_g \circ f_g^{g} \right \|_p
    $$
    where $f^g$ denotes the function $f^g(x)=f(x-g)$. 
\end{lemma}
\begin{proof}
    We have
    $$
        \left \| \sum_{g \in S} f_g \circ f_g^{g} \right \|_p^p = 
        \sum_{\alpha_1 + \dots + \alpha_p = 0} \prod_{i=1}^p \left(\sum_{g \in S} |\widehat{f}_g(\alpha_i)|^2 \cdot \chi_{\alpha_i}(g) \right) 
        \le \sum_{\alpha_1 + \dots + \alpha_p = 0} \prod_{i=1}^p \left(\sum_{g \in S} |\widehat{f}_g(\alpha_i)|^2 \right) = \left \| \sum_{g \in S} f_g \circ f_g \right \|_p^p.
    $$
\end{proof}

The merit of working with $F'$ is that it is spectrally non-negative, so all of its odd moments are non-negative. Kelley and Meka \cite{kelley2023strong} showed that if a function $f$ admits non-negative odd moments as well as deviations from uniform in the $p$-norm, then $f + 1$ must also admit strong deviations \textit{upwards} from uniform. The specific version of this lemma we will use comes from \cite{bloom2023kelley}, which adapts {\cite[Proposition 5.7]{kelley2023strong}} to work when the norm is taken with respect to a spectrally non-negative distribution $\nu$.

\begin{restatable}[{\cite[Lemma 7]{bloom2023kelley}}]{lemma}{oddmoments}\label{lem:oddmoments}
    Let $\eps \in (0,1)$ and $\nu \colon G \to \Rplus$ be some probability measure such that $\widehat{\nu} \geq 0$. Let $f \colon G \to \mathbb{R}$ be such that $\widehat{f} \geq 0$. If $\|f\|_{p(\nu)} \geq \eps$ for some $p \geq 1$, then $\|f+1\|_{p'(\nu)} \geq 1 + \eps/2$ for some integer $p'=O_{\eps}(p)$.
\end{restatable}

Note that we will only apply \Cref{lem:oddmoments} when $\eps$ is a constant, so the exact dependence on $\eps$ is inconsequential. The utility of this lemma comes when applied to the mean zero part of a function $f$, namely $f_0 \coloneqq f - \E[f]$. If $\E[f_0^p] \geq 0$ for odd $p$, then downward deviations from uniform in $f_0$ imply \textit{upwards} deviations from uniform in $f$. In particular, we will use the above lemma to exhibit upwards deviations in $F' = \sum_{g} 1_{A_g} \circ 1_{A_g}$. Later on, we will see how to use this to obtain structural information on the $\{A_g\}_{g \in G}$. Combining the previous two lemmas, we obtain the following corollary.

\begin{corollary}\label{cor:unbalancing_ff}
    Let $\eps \in (0,1)$, $F = \sum_{g \in G} 1_{A_g} \circ 1_{A_g+g}$, $F' = \sum_{g \in G} 1_{A_g} \circ 1_{A_g}$, where $\{A_g\}_{g \in G}$ is a collection of subsets of $G$. If $\|\mu_{F} -1\|_p \ge \eps$ for some integer $p \ge 1$, then $\|\mu_{F'}\|_{p'} \ge 1+\eps/2$ for some integer $p' = O_\eps(p)$.
\end{corollary}
\begin{proof}
    Let $\alpha_g = |A_g|/|G|$ and $\alpha = \sum_{g \in G} \alpha_g^2$, so that $\frac{1}{\alpha} F = \mu_F$ and $\frac{1}{\alpha} F' = \mu_{F'}$. Setting $f_g = 1_{A_g} - \alpha_g$, we have 
    $$
        \sum_{g \in G} f_g \circ f_g = \sum_{g \in G} \left(1_{A_g} - \alpha_g\right) \circ \left(1_{A_g} - \alpha_g \right) = \sum_{g \in G} 1_{A_g} \circ 1_{A_g} - \alpha_g^2 = F' - \alpha. 
    $$
    Similarly, we have
    $$
        \sum_{g \in G} f_g \circ f^g_g = \sum_{g \in G} \left(1_{A_g} - \alpha_g\right) \circ \left(1_{A_g}^g - \alpha_g \right) = \sum_{g \in G} 1_{A_g} \circ 1_{A_g+g} - \alpha_g^2 = F - \alpha.   
    $$
    If we apply \Cref{lem:decoupling_ff} with $S \coloneqq G$, this gives $\|\mu_{F'} - 1\|_{p} \geq \eps$. Since $\mu_{F'} - 1$ is spectrally non-negative, we can apply \Cref{lem:oddmoments} to obtain $\|\mu_{F'}\|_{p'} \geq 1 + \eps/2$ for some integer $p' = O_{\eps}(p)$. 
\end{proof}

\subsubsection{Dependent random choice}\label{subsec:sifting}
At this point, we can assume we are working with a sum of difference convolutions $F' = \sum_{g \in G} 1_{A_g} \circ 1_{A_g}$ which admits upwards deviations from uniform in an appropriate $p$-norm. We will want to extract information about the collection $\{A_g\}_{g \in G}$ itself from $F'$. It will be useful for us to find a ``robust witness'' to upward deviations from the mean. One can view a robust witness as being a convolution which correlates well with the set of points on which $F'$ has a large deviation upwards. We will need the following lemma by Bloom and Sisask \cite{bloom2023kelley} using dependent random choice (originally dubbed \textit{sifting} by Kelley and Meka \cite{kelley2023strong}).
\begin{lemma}[{\cite[Lemma 8]{bloom2023kelley}}]\label{lem:sifting}
    Let $\eps, \delta \in (0,1)$ and $p \geq 1$ be an integer. Let $B_1, B_2 \subseteq G$, and let $\nu = \mu_{B_1} \circ \mu_{B_2}$. For any finite set $A \subseteq G$, if 
    $$
        S = \{x \in G : (\mu_A \circ \mu_A)(x) > (1-\eps) \| \mu_A \circ \mu_A \|_{p(\nu)} \},
    $$
    then there are $M_1 \subseteq B_1$ and $M_2 \subseteq B_2$ such that
    $$
        \langle \mu_{M_1} \circ \mu_{M_2}, 1_S \rangle \geq 1 - \delta \quad\text{and}\quad\min \left(\mu_{B_1}(M_1), \mu_{B_2}(M_2) \right) \ge \left( \E[1_A] \cdot \| \mu_A \circ \mu_A \|_{p(\nu)} \right)^{2p + O_{\eps, \delta}(1)}.
    $$
\end{lemma}

As with the previous lemma (\Cref{lem:oddmoments}), we will only apply \Cref{lem:sifting} for constant $\eps$ and $\delta$, so we will not be concerned with the precise dependency on these variables. Observe that since the support of $\mu_{M_1} \circ \mu_{M_2}$ is contained in the support of $\nu$, the conclusion $\langle \mu_{M_1} \circ \mu_{M_2}, 1_S \rangle \geq 1 - \delta$ is unchanged by replacing $S$ with $S\cap L$ for any $L$ containing the support of $\nu$. This observation will be used in the proof of \Cref{lem:robust_witness_gen}.

In the 3AP setting, this lemma is enough to obtain a robust witness to $A * A$ being far from uniform. Since we are working with a collection of sets $\{A_g\}_{g \in G}$, we will encode our collection into a large set $A \subseteq G \times G$ while preserving the relevant properties and apply sifting there. Despite the lemma being stated for a single set, we are able to appropriately encode our collection to apply it.
First, we prove sifting for an arbitrary Abelian group $G$, then we apply it specifically to the case of $\F_q^n$. Later, we will apply this result to $G = \mathbb{Z}/N\mathbb{Z}$ as well. 

\begin{restatable}{lemma}{sifting}\label{lem:robust_witness_gen}
    Let $\eps, \delta \in (0,1)$ and $p \geq 1$ be an integer. Additionally, let $B_1, B_2 \subseteq G$, $\nu = \mu_{B_1} \circ \mu_{B_2}$, and $f = \sum_{g \in C} 1_{A_g} \circ 1_{A_g}$, where $\{A_g\}_{g \in C}$ is a collection of subsets of $G$ indexed by $C$ with $\sum_{g \in C} |A_g| \ge 2^{-{d}}|G||C|$. If
    \[
        S = \{x\in G : \mu_f(x) > (1 - \eps) \| \mu_f \|_{p(\nu)}\},
    \]
    then there are $M_1 \subseteq B_1$ and $M_2 \subseteq B_2$ such that
    \[
        \langle \mu_{M_1} \circ \mu_{M_2}, 1_S \rangle \geq 1 - \delta \quad\text{and}\quad\min(\mu_{B_1}(M_1), \mu_{B_2}(M_2)) \geq \left(2^{-d} \| \mu_{f}\|_{p(\nu)} \right)^{O_{\eps, \delta}(p)}. 
    \]
\end{restatable}

\begin{proof}
    Let $H = G \times G, L = \{0\} \times G$, $B_1' = \{0\} \times B_1$ and $B_2' = \{0\} \times B_2$. Additionally, let $\pi = \mu_{B_1'} \circ \mu_{B_2'}$, and note that $\pi$ is supported on $L$. We embed our collection into a single set $A = \{ (g,y) \mid g \in C, y \in A_g \}$ in $H$, and define
    \[
        S = \{x\in L : (\mu_A \circ \mu_A)(x) > (1-\eps) \|\mu_A \circ \mu_A\|_{p(\pi)}\}.
    \]
    Recall that it suffices to define $S$ for $x\in L$ (as opposed to $H$), since $L$ contains the support of $\pi$. Applying \Cref{lem:sifting}, we obtain $M_1 \subseteq B'_1$ and $M_2\subseteq B'_2$, such that
    $$
        \langle \mu_{M_1} \circ \mu_{M_2}, 1_S \rangle_{H} \geq 1 - \delta  \quad\text{and}\quad\min \left(\mu_{B'_1}(M_1), \mu_{B'_2}(M_2) \right) \ge \left( \E[1_A] \cdot \| \mu_A \circ \mu_A \|_{p(\pi)} \right)^{O_{\eps, \delta}(p)}.
    $$    
    Observe that for $x = (0,y) \in L$, we have
    \[
        (1_A \circ 1_A)(x) = \mathop{\E}_{z \in H} 1_A(z)1_A(x+z) 
        = \frac{1}{|G|^2}\sum_{g\in C} \sum_{z\in G}1_{A_{g}}(z)1_{A_{g}}(y + z)
        = \frac{1}{|G|}\sum_{g\in C} (1_{A_g}\circ 1_{A_g})(y)=\frac{1}{|G|}f(y).
    \]
    In particular, $(\mu_A \circ \mu_A)(x) = \zeta \cdot \mu_{f}(y)$, where
    \begin{equation}\label{eq:zeta_bound_int}
        \zeta = \frac{\E_G[f]}{\E_H[1_A]^{2}\cdot|G|} = \frac{|G|\cdot \sum_{g\in C}|A_g|^2}{|\sum_{g\in C}A_g|^2} \ge \frac{|G|}{|C|},
    \end{equation}
    with the final inequality by Cauchy-Schwarz. Similarly, we have
    \[
        \pi(x) = (\mu_{B_1'} \circ \mu_{B_2'})(x) = \frac{|H|^2}{|B_1'||B_2'|} \cdot (1_{B_1'} \circ 1_{B_2'})(x) = \frac{|G|^4}{|B_1||B_2|} \cdot \frac{1}{|G|}(1_{B_1} \circ 1_{B_2})(y) = |G|(\mu_{B_1}\circ\mu_{B_2})(y) = |G|\nu(y).
    \]
    Recalling that $L$ contains the support of $\pi$, we can rewrite $S$ as $\left\{(0,y) \in L : \mu_f(y) > (1-\eps)\|\mu_f\|_{p(\nu)}\right\}$.

    We would like to verify the inner product over $G$ satisfies the conclusion of the theorem. Let $\alpha_1 = |M_1|/|H|$ and $\alpha_2 = |M_2|/|H|$. Unraveling the definition of $\langle \mu_{M_1} \circ \mu_{M_2}, 1_S \rangle$, we have
    \begin{align*}
        \E_{x \in H} (\mu_{M_1} \circ \mu_{M_2})(x) 1_S(x) &= \frac{1}{|H|} \sum_{x \in H} \left(\frac{1}{|H|} \sum_{y \in H} \mu_{M_1}(y) \mu_{M_2}(x+y) \right)1_S(x) \\
        &= \frac{1}{|G|^2} \sum_{x \in H} \left(\frac{1}{|G|^2} \sum_{y \in H} \frac{1_{M_1}(y)}{\alpha_1} \cdot \frac{1_{M_2}(x+y)}{\alpha_2} \right)1_S(x) \\ 
        &= \frac{1}{|G|^2} \sum_{x \in L}  \left(\sum_{y \in L} \frac{1_{M_1}(y)}{\alpha_1 |G|} \cdot \frac{1_{M_2}(x+y)}{\alpha_2|G|}\right)1_S(x) \tag{since $M_1,M_2 \subseteq L$} \\
        &= \frac{1}{|G|} \sum_{x \in L}  \left(\frac{1}{|G|} \sum_{y \in L} \frac{1_{M_1}(y)}{\E_L[M_1]} \cdot \frac{1_{M_2}(x+y)}{\E_L[M_2]}\right)1_S(x).
    \end{align*}
    Since $L \cong G$, we can interpret $M_1,M_2,$ and $S$ as subsets of $G$ to obtain $\langle \mu_{M_1} \circ \mu_{M_2}, 1_S \rangle_G \geq 1 - \delta$. 
    Similarly, we can interpret $M_1 \subseteq B_1$ and $M_2 \subseteq B_2$. We conclude by using \Cref{eq:zeta_bound_int} to
    deduce 
    \[
        \min \left(\mu_{B_1}(M_1), \mu_{B_2}(M_2) \right) \ge \left(\E_H[1_A]\cdot \zeta \cdot \| \mu_f \|_{p(\nu)} \right)^{O_{\eps, \delta}(p)} \ge \left(2^{-d} \| \mu_f\|_{p(\nu)} \right)^{O_{\eps, \delta}(p)}. \tag*{\qedhere}
    \]
\end{proof} 

For the case where $G =\F_q^n$, we will not need the same level of generality.
\begin{corollary}\label{cor:robust_witness_ff}
    Let $\eps \in (0,1)$ and $p \geq 1$ be an integer. Additionally, let $f = \sum_{g \in \F_q^n} 1_{A_g} \circ 1_{A_g}$ where $\{A_g\}_{g \in \F_q^n}$ is a collection of subsets of $\F_q^n$ with $\sum_{g \in \F_q^n} |A_g| \geq 2^{-{d}}|\F_q^n|^2$. If $\| \mu_{f} \|_p \geq 1 + \eps$ and $S = \{x : \mu_{f}(x) > 1 + \eps / 2 \}$, then there are $M_1, M_2 \subseteq \F_q^n$ both of density at least $2^{-O_{\eps}(dp)}$ such that
    $$
        \langle \mu_{M_1} \circ \mu_{M_2}, 1_S \rangle \geq 1 - \frac{\eps}{8} .
    $$
\end{corollary}

\begin{proof}
    Apply \Cref{lem:robust_witness_gen} with $C, B_1, B_2 \coloneqq \F_q^n$, $\eps \coloneqq \eps/4$, and $\delta \coloneqq \eps/8$.
\end{proof}

\subsubsection{Almost periodicity}\label{subsec:almostPeriod}

At this point, we have obtained a robust witness to the non-uniformity of $F'$. Namely, we have two dense sets $M_1, M_2$ whose convolution $M_1 \ast M_2$ correlates well with the points where $F'$ has upwards deviations. We will want to convert our robust witness into a subspace $V$ on which our collection $\{A_g\}_{g \in \F_q^n}$ admits a collective density increment. Kelley and Meka \cite{kelley2023strong} originally used a result of Sanders \cite{sanders2012bogolyubov}; we will follow \cite{bloom2023kelley} in using the following theorem to find a large subspace $V$ which correlates well with the upwards deviations of $F'$ via the robust witness we found in the previous step. 

\begin{theorem}[{\cite[Theorem 3.2]{schoen2016roth}}]\label{thm:periodicity_ff}
    Let $\eps \in (0,1)$. If $M_1, M_2, S \subseteq \F_q^n$ are such that $M_1$ and $M_2$ both have density at least $2^{-d}$, then there is a subspace $V$ of co-dimension $O_{\eps}(d^4)$ such that
    $$
        |\langle \mu_V \ast \mu_{M_1} \ast \mu_{M_2}, 1_S \rangle - 
        \langle \mu_{M_1} \ast \mu_{M_2}, 1_S \rangle| \leq \eps.
    $$
\end{theorem}
\begin{corollary}\label{cor:periodicity_ff}
    Let $\eps \in (0,1)$, and let $M_1, M_2, S \subseteq \F_q^n$ be such that $M_1$ and $M_2$ both have density at least $2^{-d}$. If 
    \[
        \langle \mu_{M_1} \circ \mu_{M_2}, 1_S \rangle \geq 1 - \frac{\eps}{2},
    \]
    then there is a subspace $V$ of co-dimension $O_{\eps}(d^4)$ such that
    \[
        \langle \mu_V \ast \mu_{M_1} \circ \mu_{M_2}, 1_S \rangle \geq 1 - \eps.
    \]
\end{corollary}
\begin{proof}
    Apply \Cref{thm:periodicity_ff} with $\eps \coloneqq \eps/2, M_1 \coloneqq M_2, M_2 \coloneqq -M_1$, and $S \coloneqq S$. 
\end{proof}

\subsubsection{Combining steps}\label{subsec:combining}
All that remains to prove \Cref{thm:strucvspseudo_ff} (restated below) is to combine the prior steps.
\strVSpsdFF*
\begin{proof}
    If (1) does not hold, we apply \Cref{lem:nonuniform_ff} to arrive at $\|\mu_F - 1 \|_p \geq \eps/2$ for integer $p = O(d)$. \Cref{cor:unbalancing_ff} implies $\| \mu_{F'} \|_{p'} \geq 1 + \eps/4$ for some integer $p' = O_{\eps}(p)$. Let $S = \{x \in \F_q^n : \mu_{F'}(x) > 1 + \eps / 8\}$. We apply \Cref{cor:robust_witness_ff} to obtain $M_1, M_2 \subseteq \F_q^n$ both of density 
    at least $2^{-O_{\eps}(d^2)}$
    such that
    $$
        \langle \mu_{M_1} \circ \mu_{M_2}, 1_S \rangle \geq 1 - \frac{\eps}{32}.
    $$
    We then apply \Cref{cor:periodicity_ff} with $\eps \coloneqq \eps/16$ and $d\coloneqq O_{\eps}(d^2)$ to find a subspace $V$ with co-dimension at most $O_\eps(d^8)$ such that
    $$
        \langle \mu_V \ast \mu_{M_1} \circ \mu_{M_2}, 1_S \rangle \geq 1 - \frac{\eps}{16}.
    $$
    By the definition of $S$, we have
    \[
        1 + \frac{\eps}{32} \leq \left(1 + \frac{\eps}{8}\right)\left(1-\frac{\eps}{16}\right) 
        \leq \langle \mu_V \ast \mu_{M_1} \circ \mu_{M_2}, \mu_{F'} \rangle
        \leq \| \mu_{F'} \ast \mu_V\|_{\infty} \|\mu_{M_1} \circ \mu_{M_2}\|_1
        = \|\mu_{F'} \ast \mu_V\|_{\infty}. \tag*{\qedhere}
    \]
\end{proof}

\section{The integer case}
In this section, we prove the following theorem.
\begin{theorem}\label{thm:corner_bounds_Abelian}
    Let $G$ be a finite Abelian group. Any skew-corner-free subset of $G \times G$ has size at most
    \[
        \exp\left\{-\Omega\left(\log^{1/12}|G|\right)\right\}\cdot |G|^2.
    \]
\end{theorem}

Note that \Cref{thm:corner_bounds_Abelian} immediately implies \Cref{thm:corners_bounds_int} (restated below), since embedding $A \subseteq [n] \times [n]$ in $\Z/N\Z \times \Z/N\Z$ for $N=2n$ does not change the number of skew corners or (asymptotic) density, 

\mainint*

All the main ideas of \Cref{thm:corner_bounds_Abelian}'s proof are the same as those in \Cref{sec:finite_field}, although some additional technicalities arise from replacing subspaces with regular Bohr sets. (It may be helpful to revisit \Cref{sec:bohr}.) For the remainder of the discussion, we will assume we are working with $G = \mathbb{Z}/N\mathbb{Z}$, but everything we are doing works over any finite Abelian group.

We will aim to replicate what we did in the finite field case when working with $\mathbb{Z}/N\mathbb{Z}$. The main issue is that there are potentially no nontrivial subgroups to work with over $\mathbb{Z}/N\mathbb{Z}$. This ends up not being so much of an issue, since we can replace subgroups with regular Bohr sets, which approximate subgroups in many important ways. And, not unlike the case over $\F_q^n$, we will be able to argue that if a collection $\{A_g\}_{g \in G}$ contains very few skew corners, then there must be some collective density increment on a Bohr set $B$. 

In the finite field case, when we found a subspace $V$ with a collective density increment, we could restrict our attention to $V$ since $V \cong \F_q^{\dim(V)}$. Here, we will not have this luxury, since Bohr sets are not necessarily subgroups! We will have to maintain our ambient group as $\mathbb{Z}/N\mathbb{Z}$, all while restricting our sets $\{A_g\}_{g \in G}$ to smaller and smaller Bohr sets. Thus, we will work with functions defined over $\mathbb{Z}/N\mathbb{Z}$. The supports of these functions, however, will only be contained in some Bohr set $R$. So, when taking norms and inner products, we will want to zoom in on the action by taking everything with respect to a certain measure $\nu$. For the most part, $\nu$ will be some convolution of dilates of $R$, which we will be able to control since Bohr sets behave well under convolutions. Dealing with these technicalities surrounding Bohr sets and non-uniform measures forms the major differences between working over $\F_q^n$ and $\mathbb{Z}/N\mathbb{Z}$, but the overall proof structures are conceptually the same.

\subsection{Obtaining simultaneous spreadness}

As in the finite field case, we will want to perform a density increment argument on the rows and columns of $A \subseteq G \times G$ until we obtain a collection $\{A_g\}_{g \in C}$ of subsets of $R$ for regular Bohr sets $R, C \subseteq G$ which is simultaneously spread. For technical reasons, we will want $C$ to be a translate of a dilate of $R$. This will end up being necessary when we prove \Cref{lem:nonuniform} and \Cref{lem:unbalancing_int}.

Another difference between the Abelian group case and the finite field case comes from our definition of simultaneous spreadness. In particular, if our sets $\{A_g\}_{g \in C}$ are each contained in $R$, then the Bohr set $B$ on which we find a density increment must be a sub-Bohr set of $R$, which to recall we notate by $B \leq R$. The reader will notice in the proof of \Cref{lem:simul-spread} that this becomes necessary when trying to obtain $B_{\rho} \subseteq R_{\rho}$ from $B \leq R$ for some $\rho \in (0,1)$. If we try to start with an assumption of the form $B \subseteq R$, it is not true in general that $B_{\rho} \subseteq R_{\rho}$. 

Overall, the proof idea is conceptually the same as \Cref{lem:spreadness_ff} in the finite field case. The main differences come from ensuring that the index $C$ is a translate of an appropriate dilate of $R$. 

\begin{lemma}\label{lem:simul-spread}
    Let $\eps, \delta \in (0, 1), r \ge 1, d \ge 0$, and $\{A_g\}_{g \in G}$ be a collection of subsets of $G$ with $\sum_{g \in G} |A_g|^2 \geq 2^{-d} |G|^3$. Then there exists a regular Bohr set $R \subseteq G$ of rank $\rk(R) \leq O(d r / \eps)$ and density $\mu(R) \geq \delta^{O(d/\eps)}$, a shift of a regular Bohr set $C = R_{\rho} + w$ where $\rho = \left(\eps^2/rd2^d\right)^{O(d/\eps)}$ and $w\in G$, and shifts $\{x_g\}_{g \in C}$ with $x_g \in G$ such that 
    \begin{enumerate}
        \item The collection $\{ \left(A_g - x_g \right) \cap R \}_{g \in C}$ is $(r, \delta, 1 + \eps)$-simultaneously spread in $R$, and 
        \item $\sum_{g \in C} |\left(A_g - x_g \right) \cap R|^2 \geq 2^{-d} |R|^2 |C|$.
    \end{enumerate}
\end{lemma}

\begin{proof}
    We proceed by a similar density increment argument as in the finite field case (\Cref{lem:spreadness_ff}), where regular Bohr sets replace the role of vector spaces.
    
    Let $B$ be a regular Bohr set and $\{A_g\}_{g \in B_{\sigma}+w}$ be a collection of subsets of $B$ indexed by $B_{\sigma}+w$, where initially $B=G$ and $\sigma = 1$. We will assume $w = 0$ for simplicity; the proof is essentially the same for nonzero shifts. If $\{A_g\}_{g \in B_{\sigma}}$ is not $(r,\delta,1+\eps)$-simultaneously spread in $B$, then by definition, there exists a regular Bohr set $B' \leq B$ with rank $\rk(B') \leq \rk(B) + r$ and density $\mu(B') \ge \delta \mu(B)$, and shifts $\{x_g\}_{g \in B_{\sigma}}$ in $G$, such that for $A'_g \coloneqq (A_g  - x_g)\cap B'$ we have
    \begin{equation}\label{eq:not_spread_conc}
        \sum_{g \in B_{\sigma}} \left(\frac{|A'_g|}{|B'|}\right)^2 \geq (1 + \eps) \sum_{g \in B_{\sigma}} \left(\frac{|A_g|}{|B|}\right)^2.
    \end{equation}
    
    From here, the proof proceeds by filtering the index set of our collection to be some translate of a dilate of $B'$. In other words, we will restrict our index set to some translate of $B'_{\sigma'}$ where $\sigma'$ will be chosen later. 
    Consider the function $f(g) \coloneqq (|A'_g|/|B'|)^2$. Then,
    $$
        \frac{1}{|B_{\sigma}|} \sum_{g \in B_{\sigma}} \left(\frac{|A'_{g}|}{|B'|}\right)^2 = \langle \mu_{B_\sigma}, f \rangle \quad \text{ and } \quad\frac{1}{|B_{\sigma}||B'_{\sigma'}|}\sum_{g \in B_{\sigma}}\sum_{h \in B'_{\sigma'}} \left(\frac{|A'_{g + h}|}{|B'|}\right)^2 = \langle \mu_{B_{\sigma}} \ast \mu_{B'_{\sigma'}}, f \rangle.   
    $$
    Note that $\langle \mu_{B_\sigma}, f \rangle \geq 2^{-d}$, since the density of our collection has only increased from the initial assumption that $\sum_{g \in G} |A_g|^2 \geq 2^{-d}|G|^3$.
    
    We want to apply \Cref{cor:approxbohr} to $B_{\sigma}$ and $B'_{\sigma'}$.
    Set $\sigma' = c \sigma \tau$ where we will choose $\tau$ later, and $c \in [1/2,1]$ according to \Cref{lem:regular_dilation} so that $B'_{\sigma'}$ is regular. Since $B' \leq B$, \Cref{obs:subbohrdil} implies the inclusion $B'_{\sigma'} \subseteq B_{\sigma \tau} = (B_{\sigma})_{\tau}$. Applying \Cref{cor:approxbohr} gives
    \begin{align*}
        \langle \mu_{B_\sigma} \ast \mu_{B'_{\sigma'}}, f \rangle &\geq \langle \mu_{B_{\sigma}}, f \rangle - |\langle  \mu_{B_\sigma} \ast \mu_{B'_{\sigma'}} - \mu_{B_{\sigma}}, f \rangle| \\
        &\geq \langle \mu_{B_{\sigma}}, f \rangle - O(\tau \cdot \rk(B_{\sigma})) \tag{since $\|f\|_{\infty} \leq 1$} \\
        &\geq (1 - \eps/4) \langle \mu_{B_{\sigma}}, f \rangle
    \end{align*}
    for $\tau \leq c'\eps/2^d \rk(B_{\sigma})$ with constant $c' > 0$ chosen sufficiently small. 
    By our setting of $\tau$, we have
    \begin{equation}\label{eq:bohr_double_sum}
        \frac{1}{|B_{\sigma}||B'_{\sigma'}|}\sum_{g \in B_{\sigma}}\sum_{h \in B'_{\sigma'}} \left(\frac{|A'_{g + h}|}{|B'|}\right)^2 \geq (1 - \eps/4) \frac{1}{|B_{\sigma}|}\sum_{g \in B_{\sigma}} \left(\frac{|A'_g|}{|B'|}\right)^2.
    \end{equation}
    By combining \Cref{eq:not_spread_conc} with \Cref{eq:bohr_double_sum} and averaging, there is some choice of $g \in B_{\sigma}$ so that
    $$
        \frac{1}{|B'_{\sigma'}|} \sum_{h \in B'_{\sigma'}} \left(\frac{|A'_{g+h}|}{|B'|}\right)^2 \geq (1-\eps/4)(1+\eps) \frac{1}{|B_{\sigma}|}\sum_{g \in B_{\sigma}} \left(\frac{|A_g|}{|B|}\right)^2 \geq (1+\eps/2) \frac{1}{|B_{\sigma}|}\sum_{g \in B_{\sigma}} \left(\frac{|A_g|}{|B|}\right)^2.
    $$
    
    If our new collection $\{A'_h\}_{h \in B'_{\sigma'}+g}$ is not $(r, \delta, 1 + \eps)$-simultaneously spread in $B'$, we can iterate this procedure. 
    After $i$ iterations, we obtain Bohr sets $B^{(i)}, C^{(i)}$ and a collection $\{A_g^{(i)}\}_{g \in C^{(i)}}$ such that:
    \begin{enumerate}
        \item $B^{(i)}$ is a regular sub-Bohr set of $B^{(i-1)}$ of rank at most $i \cdot r$ and density at least $\delta^i$.
        \item Each $C^{(i)}$ is a translate of $B^{(i)}_{\sigma_i}$ for $\sigma_i \ge (\tau/2)^i$. 
        \item Each $A_g^{(i)}$ equals $(A_g^{(i-1)} - x_g^{(i)}) \cap B^{(i)}$ for some shift $x_g^{(i)}$.
        \item We have
        $$
            \frac{1}{|B^{(i)}_{\sigma_i}|}\sum_{g \in C^{(i)}} \left(\frac{|A_{g}^{(i)}|}{|B^{(i)}|}\right)^2 \geq (1 + \eps/2)^i \frac{1}{|G|} \sum_{g \in G} \left(\frac{|A_g|}{|G|} \right)^2.
        $$
    \end{enumerate}
    The process must end within $O (d / \eps )$ iterations, since the density on the right-hand side of (4) cannot exceed 1. This implies that $\rk(B^{(i)}) \le O(r d /\eps)$ for all $i$, so we can obtain the claimed bound on the final $\sigma_i$ by ultimately choosing $\tau =c'' \eps^2/(rd 2^d)$ for some constant $c''>0$. 
\end{proof}

\subsection{Using simultaneous spreadness}

Assuming we have a simultaneously spread collection, it remains to show the corresponding set must have roughly as many skew corners as a random set of the same density. This can be viewed as a ``structure vs. pseudo-randomness'' result, where we say that either a set $A \times A \subseteq G$ has roughly as many skew-corners as a random set of the same density, or $A$ admits a type of density increment on a structured set. 

\begin{restatable}{theorem}{strVSpsdAbelian}\label{thm:strucvspseudo}
    There exists an absolute constant $c > 0$ such that the following holds. Let $\eps \in (0,1)$ and $R \subseteq G$ be a regular Bohr set of rank $r$. Furthermore, let $\{A_g\}_{g \in C}$ be a collection of subsets of $R$ indexed by a regular Bohr set $C = R_{\sigma}$ with $\sigma \le c\eps/r2^d$, where $\sum_{g \in C} |A_g| \ge 2^{-d} |R| |C|$.
    Moreover, let $F = \sum_{g \in C} 1_{A_g} \circ 1_{A_g+g}$ and $F' = \sum_{g \in C} 1_{A_g} \circ 1_{A_g}$, and let $D(x) = |A_x|$ when $x \in C$ and 0 otherwise. Then either
    \begin{enumerate}
        \item 
        $|\langle \mu_F, \mu_D \rangle - \mu(R)^{-1}| \leq \eps \mu(R)^{-1}$, or
        \item There exists a sub-Bohr set $B \leq R$ with $\rk(B) \le \rk(R)+O_{\eps}(d^8)$ and $\mu(B) \geq \delta \mu(R)$ 
        such that $$\| \mu_{F'} \ast \mu_B \|_{\infty} \geq (1 + \eps/80)\mu(R)^{-1},$$ where $\delta = \exp\left(-O_{\eps}(r \log(1/\sigma) + rd^2 + rd\log r + d^{10})\right)$. 
    \end{enumerate}
\end{restatable}

A quick aside on normalization: say for a moment that the $\{A_g\}_{g \in C}$ are random subsets of $R$, and imagine we were in the perfectly structured situation that $R + R = R$. Then, we would expect $\mu_F$ to be a very uniform function supported on $R + C$. Since $C$ is a sufficiently narrow dilate of $R$, by regularity we can think of the distribution $R+C$ as being ``close'' to the uniform distribution on $R$. Thus, we are left with $\mu_F$ which is roughly uniform on $R$, and zero everywhere else. Since $\mu_F$ has mean 1 on all of $G$, we can really think of $\mu_F$ as being a uniform function taking values close to $\mu(R)^{-1}$ on $R$, and zero everywhere else. Making these structured statements into approximate ones will follow from Bohr set machinery. 

For our application (\Cref{thm:corner_bounds_Abelian}), we will instantiate $\sigma$ to be the value given by \Cref{lem:simul-spread}, which is smaller than the bound in the hypothesis of \Cref{thm:strucvspseudo}. This does not affect our final bound.

\begin{proof}[Proof of \Cref{thm:corner_bounds_Abelian} from \Cref{thm:strucvspseudo}] 
    Let $A \subseteq G \times G$ be skew-corner-free with $|A| = 2^{-d}|G|^2$, and define the collection $\{A_g\}_{g\in G}$ by $A_x = \{y : (x,y) \in A\}$. By Cauchy-Schwarz, we have $\sum|A_g|^2 \ge 2^{-2d}|G|^3$, so applying \Cref{lem:simul-spread} with $\{A_g\}_{g\in G}$, $\eps\coloneqq 1/80$, $\delta\coloneqq \exp(-cd^{11})$, $d\coloneqq 2d$, and $r \coloneqq c'd^8$ for constants $c,c' > 0$ to be determined later yields a collection of sets $\{\left(A_g - x_g \right) \cap R \}_{g \in C+w}$, where $R\subseteq G$ is a regular Bohr set of rank at most $O(d^9)$ and density $\mu(R) \geq 2^{-O(d^{12})}$, $C = R_{\rho}$ for $\rho = 2^{-O(d^2)}$ is a regular Bohr set, and $\{x_g\}_{g \in C+w}, w \in G$ are shifts. Observe that $\mu(C) \ge 2^{-O(d^{12})}$ by \Cref{lem:bohrsize}.

    For brevity, define $A'_g \coloneqq (A_{g+w} - x_{g+w}) \cap R$ for $g\in C$ and $A' \coloneqq \{(g,y): g \in C, y \in A'_g\} \subseteq C \times R$. Note that by \Cref{obs:shift}, $A$ has at least as many skew corners as $A'$, $|A'| \ge 2^{-2d}|R||C|$, and $\{A'_g\}_{g \in C}$ is $(r,\delta,1+\eps)$-simultaneously spread in $R$. 

    Apply \Cref{thm:strucvspseudo} with $\eps \coloneqq 1/2$ to the new collection, and set $c'$ large enough for $r$ to exceed the rank bound in the second conclusion. If $|\langle \mu_F, \mu_D \rangle - \mu(R)^{-1}| \leq \tfrac{1}{2} \mu(R)^{-1}$, \Cref{lem:normalized_corners_count} implies $A'$ contains at least
    \[
        \frac{\mu(R)^{-1}}{2}\cdot\frac{|A'|^3}{|G|^2} \ge 2^{-6d-1}\mu(R)^2\mu(C)^3 |G|^4 \ge 2^{-O(d^{12})}|G|^4
    \]
    skew corners. Since there are at most $|G|^3$ trivial skew corners in $A'$, we are done after rearrangement. 
    Otherwise, there exists a sub-Bohr set $B \leq R$ with $\rk(B) \le \rk(R) + r$ and $\mu(B) \geq \delta' \mu(R)$ 
    such that $$\| \mu_{F'} \ast \mu_B \|_{\infty} \geq \left(1 + \frac{1}{160}\right)\mu(R)^{-1},$$ where $\delta' = \exp\left(-O(d^{11})\right)$. By \Cref{lem:inftospread}, this contradicts the fact that $\{A'_{g}\}_{g \in C}$ is $(r,\delta,1+1/80)$-simultaneously spread in $R$, assuming $c$ is large enough to ensure $\delta < \delta'$.
\end{proof}

\subsubsection{Non-uniformity}

The following lemma shows that if no corners are present, then we obtain some non-uniformity on our collection of sets in an appropriate $p$-norm. Here, it will be critical that $C$ is contained in a sufficiently narrow dilate of $R$ so that we can apply \Cref{lem:BohrConv}. This combined with an application of H\"{o}lder's inequality constitutes most of the proof.

\begin{lemma}\label{lem:nonuniform}
    There exists an absolute constant $c > 0$ such that the following holds. Let $\eps > 0$ and $\{A_g\}_{g \in C}$ be a collection of subsets of a regular Bohr set $R$ of rank $r$, indexed by a regular Bohr set $C = R_{\sigma}$ with $\sigma \le c\eps/r 2^d$, where $\sum_{g \in C} |A_g| \ge 2^{-d} |R| |C|$. Define $\alpha_g = |A_g|/|G|$ and $\alpha = \sum_{g \in C} \alpha_g^2$. Additionally, let $D(x) = |A_x|$ when $x \in C$ and 0 otherwise, and let $f = \sum_{g \in C} 1_{A_g} \circ 1_{A_g+g}$. Then either
    \begin{enumerate}
        \item 
        $|\langle \mu_f, \mu_D \rangle - \mu(R)^{-1}| \leq \eps \mu(R)^{-1}$ or
        \item
        There is some even integer $p = O(d)$ such that 
        $$
            \left \|\frac{1}{\alpha} \sum_{g \in C} \left(1_{A_g} - \alpha_g \mu_R \right) \circ \left(1_{A_g+g} - \alpha_g \mu_{R+g}\right) \right \|_{p(\mu_{C})} \geq \frac{\eps}{2}\mu(R)^{-1}.
        $$
    \end{enumerate}
\end{lemma}

\begin{proof}
    Our proof follows closely \cite[Proposition 20]{bloom2023kelley}. Assume the first conclusion does not hold, and note that 
    \begin{align}
        \left \langle\frac{1}{\alpha} \sum_{g \in C} \left(1_{A_g} - \alpha_g \mu_R \right) \circ \left(1_{A_g+g} - \alpha_g \mu_{R+g}\right), \mu_D \right \rangle &= \left(\frac{1}{\alpha} \sum_{g \in C} \left \langle  1_{A_g} \circ 1_{A_g+g}, \mu_D \right \rangle - \mu(R)^{-1}\right) \label{eq:first}\\
        &- \left(\frac{1}{\alpha} \sum_{g \in C} \alpha_g \langle  1_{A_g} \circ  \mu_{R+g}, \mu_D \rangle - \mu(R)^{-1}\right) \label{eq:second} \\
        &- \left(\frac{1}{\alpha} \sum_{g \in C} \alpha_g \langle \mu_R \circ 1_{A_g+g}, \mu_D \rangle - \mu(R)^{-1}\right) \label{eq:third}\\
        &+ \left(\frac{1}{\alpha} \sum_{g \in C} \alpha_g^2 \langle \mu_R \circ \mu_{R+g}, \mu_D \rangle - \mu(R)^{-1}\right). \label{eq:fourth}
    \end{align}
We will bound the last three terms in magnitude, then apply H\"{o}lder's inequality to the first term (\ref{eq:first}) to obtain the second conclusion. Consider the second term (\ref{eq:second}); we have
\begin{align}
    \frac{1}{\alpha} \sum_{g \in C} \alpha_g \langle  1_{A_g} \circ  \mu_{R+g}, \mu_D \rangle - \mu(R)^{-1} &
    =  \frac{1}{\alpha} \sum_{g \in C} \alpha_g \langle  1_{A_g} , \mu_{R+g} \ast \mu_{-D} \rangle - \mu(R)^{-1} \nonumber \\
    &=  \frac{1}{\alpha} \sum_{g \in C} \alpha_g \langle  1_{A_g} , \mu_{R} \ast \mu_{-D}^{g} \rangle - \mu(R)^{-1} \nonumber \\
    &= \left(\frac{1}{\alpha} \sum_{g \in C} \alpha_g \langle 1_{A_g}, \mu_R \rangle\right) + \left(\frac{1}{\alpha} \sum_{g \in C} \alpha_g \langle  1_{A_g} , \mu_R \ast \mu_{-D}^{g} - \mu_R \rangle\right) - \mu(R)^{-1}. \label{eq:second_group}
\end{align}
First, since $A_g \subseteq R$, we know that $\langle 1_{A_g}, \mu_R \rangle$ = $\alpha_g \mu(R)^{-1}$. Summing over $g \in C$ gives
$$
\frac{1}{\alpha} \sum_{g \in C} \alpha_g \langle 1_{A_g}, \mu_R \rangle = \frac{1}{\alpha} \sum_{g \in C} \alpha^2_g \cdot \mu(R)^{-1} = \mu(R)^{-1}, 
$$
so the first and third terms of (\ref{eq:second_group}) cancel out. We next bound the second term of (\ref{eq:second_group}).
Since $\mu_D$ is supported on $C = R_{\sigma}$ and $C$ is symmetric, we have $\mu_{-D}$ is supported on $C$ as well. As $g \in C$, we can infer that $\mu_{-D}^{g}$ is supported on $C + C \subseteq R_{2\sigma}$. The last point to note is that $\frac{1}{\alpha} \left( \sum_{g \in C} \alpha_g \right) \leq 2^{d} \mu(R)^{-1}$ by Cauchy-Schwarz. We will now use the assumption that $\sigma\le c\eps/2^d \rk(R)$ for a small enough constant $c>0$. By \Cref{cor:approxbohr}, we know 
\begin{align*}
    \left|\frac{1}{\alpha} \sum_{g \in C} \alpha_g \langle  1_{A_g} , \mu_R \ast \mu_{-D}^g - \mu_R \rangle \right| &\leq \frac{1}{16 \alpha} \sum_{g \in C}  \alpha_g \eps 2^{-d} \leq \frac{\eps}{16} \mu(R)^{-1}.
\end{align*}
Similar reasoning bounds the third (\ref{eq:third}) and fourth (\ref{eq:fourth}) terms by the same quantity. 
Finally, let $f_0 = \sum_{g \in C} \left(1_{A_g} - \alpha_g \mu_R \right) \circ \left(1_{A_g+g} - \alpha_g \mu_{R+g}\right)$. Combining the above bounds gives
$$
    \left| \langle \mu_f, \mu_D \rangle - \mu(R)^{-1} - \langle \mu_{f_0}, \mu_D \rangle \right| \le \frac{\eps}{4} \mu(R)^{-1}.
$$
For simplicity, define $D'(x) = D(x)/|R|$, so that $\|D'\|_{\infty} \leq 1$. Note that $\mu_D = \mu_{D'}$. If the first conclusion fails, and using the fact that $D'$ is supported on $C$, we have
$$
    \frac{3}{4} \eps \mu(R)^{-1} \le |\langle \mu_{f_0}, \mu_{D} \rangle| = |\langle \mu_{f_0}, \mu_{D'} \rangle| = \E_C[D']^{-1} \cdot |\langle \mu_{f_0}, D' \rangle_{\mu_C}| .
$$
An application of H\"{o}lder's inequality, for any $p \geq 1$, gives 
$$
    \left| \langle \mu_{f_0}, D' \rangle_{\mu_C} \right| \le \|\mu_{f_0}\|_{p(\mu_C)} \cdot \|D'\|_{\left(\frac{p}{p-1}\right)(\mu_C)} \le  \|\mu_{f_0}\|_{p(\mu_C)} \cdot \|D'\|_{\infty} \cdot \|D'\|_{1(\mu_C)}^{1-1/p}.
$$
We have $\|D'\|_{1(\mu_C)} = \E_C[D']^{-1} \geq 2^{-d}$. Combining the above inequalities gives
$$
    2^{d/p} \| \mu_{f_0} \|_{p(\mu_C)} \geq \left| \langle \mu_{f_0}, \mu_{D'} \rangle \right| \ge \frac{3}{4} \eps \mu(R)^{-1}.
$$
Choosing $p$ an even integer so that $2^{d/p} \geq 3/2$ gives the result.
\end{proof}

\subsubsection{Unbalancing}

At this point, we have a non-uniformity assumption on the function
$$
    F_0 = \sum_{g \in C} \left(1_{A_g} - \alpha_g \mu_R \right) \circ \left(1_{A_g+g} - \alpha_g \mu_{R+g}\right),
$$
and we would like to switch to working with the function $F' = \sum_{g \in C} 1_{A_g} \circ 1_{A_g}$. The reason for this is mostly technical; $F'$ has non-negative Fourier coefficients, while $F_0$ in general may not. There are minor complications that arise when working with non-uniform measures, which accounts for the main differences between the following lemma and the corresponding one in the finite field case. 

The goal of this section is to prove the following: 

\begin{restatable}{lemma}{unbalancingInt}\label{lem:unbalancing_int}
    There exists an absolute constant $c > 0$ such that the following holds. Let $\eps \in (0,1)$ and $p \geq 2$ an even integer. Furthermore, let $\{A_g\}_{g \in C}$ be a collection of subsets of a regular Bohr set $R$ of rank $r$ indexed by a regular Bohr set $C = R_{\sigma}$ for $\sigma \leq c \eps/r2^d$. Moreover, let $C', C'' \subseteq C_{\tau}$ both be regular Bohr sets for any $\tau \leq c r ^{-1}$, and define $\alpha_g = |A_g|/|G|$ and $\alpha = \sum_{g \in C} \alpha_g^2$. Then 
    $$
        \left \|\frac{1}{\alpha} \sum_{g \in C} \left(1_{A_g} - \alpha_g \mu_R \right) \circ \left(1_{A_g+g} - \alpha_g \mu_{R+g}\right) \right \|_{p(\mu_{C})} \geq \eps \mu(R)^{-1}
    $$
    implies that
    $$
        \left \| \frac{1}{\alpha} \sum_{g \in C} 1_{A_g} \circ 1_{A_g} \right \|_{p'(\mu_{C'} \ast \mu_{C'' - x})} \geq \left(1+\frac{\eps}{8}\right) \mu(R)^{-1}
    $$
    for some shift $x \in G$ and integer $p' = O_{\eps}(p)$. 
\end{restatable}

To prove \Cref{lem:unbalancing_int}, we would like to remove the ``shifts'' in $F_0$ while maintaining the non-uniformity assumption. As a first step, we would like to switch to working with the function 
$$
    F'_0 = \sum_{g \in C} \left(1_{A_g} - \alpha_g \mu_R \right) \circ \left(1_{A_g} - \alpha_g \mu_{R}\right),
$$
which enjoys the property of being spectrally non-negative. For technical reasons, we will want to start working with a non-uniform measure which is also spectrally non-negative, so that we can eventually apply \Cref{lem:oddmoments}. The following lemma accomplishes both of these goals with little loss in parameters. 

\begin{lemma}\label{lem:settoconvolve}
    There exists an absolute constant $c > 0$ such that the following holds. Let $p \geq 2$ be an even integer and $\{f_g\}_{g \in S}$ be a collection of real-valued functions on $G$ indexed by some set $S \subseteq G$. Additionally, let $B \subseteq G$ and $B', B'' \subseteq B_{\tau}$ all be regular Bohr sets for $\tau \le c\cdot\rk(B)^{-1}$. Then
    $$
        \left \| \sum_{g \in S} f_g \circ f_g \right \|_{p(\mu_{B'} \circ \mu_{B'} \ast \mu_{B''} \circ \mu_{B''})} \geq \frac{1}{2} \left \| \sum_{g \in S} f_g \circ f_g^{g} \right \|_{p(\mu_B)}
    $$
    where $f^g$ denotes the function $f^g(x)=f(x-g)$.
\end{lemma}

\begin{proof}
For simplicity, we will define 
$$
    F_1 \coloneqq \sum_{g \in S} f_g \circ f_g^{g} \quad\text{and}\quad F_2 \coloneqq \sum_{g \in S} f_g \circ f_g.
$$
It follows by linearity that for $\gamma \in \widehat{G}$, we have
$$
    \widehat{F_1}(\gamma) = \sum_{g \in S} |\widehat{f_g}(\gamma)|^2 \gamma(g) \quad\text{and}\quad\widehat{F_2}(\gamma) = \sum_{g \in S} |\widehat{f_g}(\gamma)|^2,
$$
and by the triangle inequality, we have $|\widehat{F_1}(\gamma)| \leq |\widehat{F_2}(\gamma)|$.

From here, the proof is almost identical to \cite[Proposition 19]{bloom2023kelley}. Let $\nu = \mu_{B'} \circ \mu_{B'} \ast \mu_{B''} \circ \mu_{B''}$, and observe that $\nu$ is supported on $B_{\rho}$ for $\rho=4 \tau$. Additionally, note that $\nu$ is spectrally non-negative, since $\widehat{\nu}(\gamma) = \widehat{\mu_{B'}}(\gamma)^2 \widehat{\mu_{B''}}(\gamma)^2$.
Apply \Cref{lem:boundingWithDilation} to obtain 
\begin{align*}
    \mu_B &\leq 2\mu_{B_{1 + \rho}} \ast \nu.
\end{align*}
Expanding, we have
\begin{align*}
    \|F_1 \|_{p(\mu_B)}^{p} &= \E_{x \in G} \mu_B(x) F_1(x)^p \\
    &\leq 2 \E_{x \in G} \left( \mu_{B_{1+\rho}} \ast \nu \right)(x) F_1(x)^p \\
    &= 2 \E_{t \in B_{1+\rho}} \E_{x \in G} \nu(x-t) F_1(x)^p.
\end{align*}
Recall that $\widehat{f^p} = \widehat{f}^{(p)}$, where $\widehat{f}^{(p)}$ is the $p$-fold convolution of $\widehat{f}$.
By averaging, there exists some $t \in B_{1 + \tau}$ such that 
\begin{align*}
    \| F_1 \|_{p(\mu_B)}^{p} &\leq 2 \E_{x \in G} \nu(x-t) F_1(x)^p \\
    &= 2 \sum_{\gamma \in \widehat{G}} \widehat{\nu}(\gamma) \gamma(-t) \left( \widehat{F_1}\right)^{(p)}(\gamma) \\
    &\leq 2 \sum_{\gamma \in \widehat{G}} \widehat{\nu}(\gamma)\left( | \widehat{F_1} | \right)^{(p)}(\gamma) \\
    &\leq 2 \sum_{\gamma \in \widehat{G}} \widehat{\nu}(\gamma)\left( | \widehat{F_2} | \right)^{(p)}(\gamma) \\
    &= 2 \| F_2 \|_{p(\nu)}^p. \tag*{\qedhere}
\end{align*}
\end{proof}

From here we can instantiate the previous lemma to obtain the following corollary. 

\begin{corollary}\label{cor:usedecoupling}
    There exists an absolute constant $c > 0$ such that the following holds. Let $\eps \in (0,1)$ and $p \geq 2$ an even integer. Furthermore, let $\{A_g\}_{g \in C}$ be a collection of subsets of $R \subseteq G$ indexed by a regular Bohr set $C$. Define $\alpha_g = |A_g|/|G|$ and $\alpha = \sum_{g \in C} \alpha_g^2$. Moreover, let $C', C'' \subseteq C_{\tau}$ be regular Bohr sets for some $\tau \le c\cdot \rk(C)^{-1}$, and let $\nu = \mu_{C'} \circ \mu_{C'} \ast \mu_{C''} \circ \mu_{C''}$. Then 
    $$
        \left \|\frac{1}{\alpha} \sum_{g \in C} \left(1_{A_g} - \alpha_g \mu_R \right) \circ \left(1_{A_g+g} - \alpha_g \mu_{R+g}\right) \right \|_{p(\mu_{C})} \geq \eps \mu(R)^{-1}
    $$
    implies that
    $$
        \left \| \frac{1}{\alpha} \sum_{g \in C} \left(1_{A_g} - \alpha_g \mu_R \right) \circ \left(1_{A_g} - \alpha_g \mu_R \right) \right \|_{p(\nu)} \geq \frac{1}{2} \eps \mu(R)^{-1}.
    $$
\end{corollary}

\begin{proof}
    Apply \Cref{lem:settoconvolve} with $B \coloneqq C$, $S \coloneqq C$, $B' \coloneqq C'$, $B'' \coloneqq C''$, and $f_g \coloneqq 1_{A_g} - \alpha_g \mu_R$. We have 
    $$
        \sum_{g \in C} f_g \circ f_g = \sum_{g \in C} \left(1_{A_g} - \alpha_g\mu_R\right) \circ \left(1_{A_g} - \alpha_g\mu_R \right) 
    $$
    and
    \[
        \sum_{g \in C} f_g \circ f^g_g = \sum_{g \in C} \left(1_{A_g} - \alpha_g\mu_R\right) \circ \left(1_{A_g}^g - \alpha_g\mu_R^g \right) = \sum_{g \in C} \left(1_{A_g} - \alpha_g \mu_R \right) \circ \left(1_{A_g+g} - \alpha_g \mu_{R+g}\right). \tag*{\qedhere}
    \]
\end{proof}

The following lemma lets us move from a non-uniformity assumption on a function $F'_0$ whose mean is zero to an assumption on $F'$. This will be critical when we want to apply ``sifting'' later (see \Cref{sec:sifting_int}). The lemma applies \Cref{lem:oddmoments} after making some estimates using Bohr sets and \Cref{lem:BohrConv}. 

\begin{lemma}\label{thm:usingoddmoments}
    There exists an absolute constant $c > 0$ such that the following holds. Let $\eps \in (0,1)$ and $p \geq 2$ be an integer. Additionally, let $\{A_g\}_{g \in C}$ be a collection of subsets of a regular Bohr set $R$ of rank $r$ indexed by a regular Bohr set $C$ with $\sum_{g \in C} |A_g| \geq 2^{-d} |R| |C|$. Define $\alpha_g = |A_g|/|G|$ and $\alpha = \sum_{g \in C} \alpha_g^2$, and
    let $\nu \colon G \to \mathbb{R}_{\geq 0}$ be a probability measure supported on $R_{\sigma}$ for $\sigma \le c\eps/r2^d$ such that $\widehat{\nu} \geq 0$. If 
    $$
        \left \| \frac{1}{\alpha} \sum_{g \in C} \left( 1_{A_g} - \alpha_g \mu_R \right) \circ \left( 1_{A_g} - \alpha_g \mu_R \right) \right \|_{p(\nu)} \geq \eps \mu(R)^{-1},
    $$
    then there is an integer $p' = O_{\eps}(p)$ such that
    $$
        \left \| \frac{1}{\alpha} \sum_{g \in C} 1_{A_g} \circ 1_{A_g} \right \|_{p'(\nu)} \geq \left(1 + \eps / 4 \right) \mu(R)^{-1}.
    $$
\end{lemma}

\begin{proof}
    The proof follows \cite[Proposition 18]{bloom2023kelley}. We write 
    $$
        f = \frac{1}{\alpha} \sum_{g \in C} \alpha_g (1_{A_g} \circ  \mu_{R})  + \alpha_g (\mu_R \circ 1_{A_g}) - \alpha_g^2 (\mu_R \circ \mu_R).
    $$
    For any $p' \geq 1$,
    \begin{align*}
        \left \| \frac{1}{\alpha} \sum_{g \in C} 1_{A_g} \circ 1_{A_g} \right \|_{p'(\nu)} &= \left \| \frac{1}{\alpha} \sum_{g \in C} \left( 1_{A_g} - \alpha_g \mu_R \right) \circ \left( 1_{A_g} - \alpha_g \mu_R \right) + \mu(R)^{-1} + f - \mu(R)^{-1} \right \|_{p'(\nu)} \\
        &\geq \left \| \frac{1}{\alpha} \sum_{g \in C} \left( 1_{A_g} - \alpha_g \mu_R \right) \circ \left( 1_{A_g} - \alpha_g \mu_R \right) + \mu(R)^{-1} \right \|_{p'(\nu)} - \left \| f - \mu(R)^{-1} \right \|_{p'(\nu)}. 
    \end{align*}
    First, we use spectral positivity and \Cref{lem:oddmoments} to obtain that for $p'=O_{\eps}(p)$ it holds
    $$
        \left \| \frac{1}{\alpha} \sum_{g \in C} \left( 1_{A_g} - \alpha_g \mu_R \right) \circ \left( 1_{A_g} - \alpha_g \mu_R \right) + \mu(R)^{-1} \right \|_{p'(\nu)} \geq \left(1 + \frac{\eps}{2} \right) \mu(R)^{-1}.
    $$
    Next, we use \Cref{cor:approxbohr} to bound the error term. For $x \in \text{supp}(\nu)$, we have 
    \begin{align*}
        \frac{1}{\alpha}\sum_{g \in C} \alpha_g | (1_{A_g} \circ \mu_R)(x) - (1_{A_g} \circ \mu_R) (0)| &= \frac{1}{\alpha}\sum_{g \in C} \alpha_g \left| \langle 1_{A_g}, \mu_{R-x} - \mu_R \rangle \right| \\
        &= \frac{1}{\alpha}\sum_{g \in C} \alpha_g \left| \langle 1_{A_g}, \mu_{R} \ast 1_{-x} - \mu_R \rangle \right| \\
        &\leq \frac{1}{\alpha}\sum_{g \in C} \alpha_g \cdot \sigma \text{ }\rk(R) \\
        &= \frac{\eps}{4}  \mu(R)^{-1},
    \end{align*}
    where the implicit constant in $\sigma$ is set sufficiently small. Similar reasoning gives the same bounds for the other terms in $f$. We have
    $$
        \| f - \mu(R)^{-1} \|_{p'(\nu)} \leq \| f - \mu(R)^{-1} \|_{L^{\infty}\left( \text{supp}(\nu) \right)} \leq \frac{1}{4} \mu(R)^{-1}.
    $$
    This concludes the proof.
\end{proof}

Putting everything together, we can now prove the main result of this section (restated below). The choice of measure $\mu_{C'} \ast \mu_{C''-x}$ in the conclusion may seem a bit mysterious, but it will be necessary when we apply (in \Cref{cor:robust_witness_int}) that the measure we are working with is a convolution of two sets. 

\unbalancingInt*

\begin{proof}
    We first apply \Cref{cor:usedecoupling} to infer that 
    $$
        \left \|\frac{1}{\alpha} \sum_{g \in C} \left(1_{A_g} - \alpha_g \mu_R \right) \circ \left(1_{A_g} - \alpha_g \mu_R \right) \right \|_{p(\nu)} \geq \frac{\eps}{2} \mu(R)^{-1}.
    $$
    Then \Cref{thm:usingoddmoments} implies that
    $$
        \left \| \frac{1}{\alpha} \sum_{g \in C} 1_{A_g} \circ 1_{A_g} \right \|_{p'(\nu)}\geq (1 + \eps/8) \mu(R)^{-1}
    $$
    for some integer $p' = O_{\eps}(p)$ and $\nu = \mu_{C'} \circ \mu_{C'} \ast \mu_{C''} \circ \mu_{C''}$. By averaging, for some choice of $x \in C' + C''$, we must have
    \[
        \left \| \frac{1}{\alpha} \sum_{g \in C} 1_{A_g} \circ 1_{A_g} \right \|_{p'(\mu_{C'} \ast \mu_{C'' - x})}\geq (1 + \eps/8) \mu(R)^{-1}. \tag*{\qedhere}
    \]
\end{proof}

\subsubsection{Dependent random choice}\label{sec:sifting_int}

Here, we perform the analogous step to what Kelley and Meka call ``sifting,'' relying on \Cref{lem:robust_witness_gen}.

\begin{corollary}\label{cor:robust_witness_int}
    Let $\eps \in (0,1)$ and $p \geq 1$ be an integer. Furthermore, let $\{A_g\}_{g \in C}$ be a collection of subsets of a set $R\subseteq G$ indexed by a set $C\subseteq G$. Assume $\sum_{g \in C} |A_g| = 2^{-d} |R| |C|$.
    Moreover, let $B, B' \subseteq G$ and $f = \sum_{g \in C} 1_{A_g} \circ 1_{A_g}$. If 
    $$
        \left \| \mu_f \right \|_{p(\mu_{B} \ast \mu_{B'})} \geq (1 + \eps) \mu(R)^{-1}
    $$
    and 
    $$
        S = \left\{x \in G : \mu_{f}(x) > (1-\eps/2) \| \mu_{f} \|_{p(\mu_{B} \ast \mu_{B'})} \right\},
    $$
    then there exists $M_1 \subseteq B$, $M_2 \subseteq B'$ such that 
    \[
        \langle \mu_{M_1} \circ \mu_{M_2}, 1_{S} \rangle \geq 1 - \frac{\eps}{4} \quad\text{and}\quad \min\left(\mu_{B}(M_1), \mu_{B'}(M_2) \right) \geq 2^{-O_{\eps}(dp)}.
    \]
\end{corollary}

\begin{proof}
    Apply \Cref{lem:robust_witness_gen} to $R$ and $C$ with $B_1 \coloneqq B$, $B_2 \coloneqq -B'$, $d \coloneqq d-\log(\mu(R))$, $\eps \coloneqq \eps/2$, and $\delta \coloneqq \eps/4$.
\end{proof}

\subsubsection{Almost periodicity}\label{sec:almost_period_int}

\Cref{cor:robust_witness_int} allowed us to turn a non-uniformity assumption on $F'$ into dense sets $M_1, M_2$ which are robust witnesses to upwards deviations in $F'$. The following result allows us to obtain a density increment on some translate of a sufficiently large Bohr set. 

\begin{theorem}[{\cite[Theorem 5.4]{schoen2016roth}, see also \cite[Theorem 17]{bloom2023kelley}}]\label{thm:almostperiod}
    Let $\eps > 0$ and $B, B' \subseteq G$ be regular Bohr sets of rank $r$. Suppose that $M_1 \subseteq B$ with $\mu_B(M_1) = 2^{-d_1}$ and $M_2 \subseteq B'$ with $\mu_{B'}(M_2) = 2^{-d_2}$. Let $S$ be any set with $|S| \leq 2 |B|$. Then there is a regular Bohr set $B'' \leq B'$ of rank at most 
    $$
        \rk(B'') \le r + O_{\eps}(d_1^3 d_2)
    $$
    and density
    $$
        \mu(B'') \geq \exp\left(-O_{\eps}((d_1^3 d_2 + r) (d_1 + d_2 + \log r ) ) \right)\mu(B')
    $$
    such that
    $$
        |\langle \mu_{B''} \ast \mu_{M_1} \circ \mu_{M_2}, 1_S \rangle - \langle \mu_{M_1} \circ \mu_{M_2}, 1_S \rangle | \leq \eps.
    $$
\end{theorem}

Note that the Bohr set $B''$ is a sub-Bohr set of $B'$ in the conclusion of the above theorem. We will apply \Cref{thm:almostperiod} to obtain the corollary below. The choice of dilates for $C', C''$ in the hypothesis are simply to ensure that $|C' + C''| \leq 2|C'|$ via regularity. 

\begin{corollary}\label{cor:density_increment}
Let $\eps \in (0,1)$ and $p \geq 1$ an integer. Furthermore, let $\{A_g\}_{g \in C}$ be a collection of subsets of a Bohr set $R$ of rank $r$ indexed by a regular Bohr set $C$, and assume $\sum_{g \in C} |A_g| = 2^{-d} |R| |C|$. Let $C'$ be a rank $r$ regular Bohr set, and let $C'' = C'_{\sigma}$ where $\sigma = O(r^{-1})$ is chosen to ensure $C''$ is regular. Moreover, let $M_1 \subseteq C'$ and $M_2 \subseteq C''-x$ for some shift $x \in G$ with density guarantees $\mu_{C'}(M_1), \mu_{C''-x}(M_2) \geq 2^{-O_{\eps}(dp)}$. Let $S \subseteq G$ be any set. If
    $$
        \langle \mu_{M_1} \circ \mu_{M_2}, 1_{S} \rangle \geq 1 - \eps/2,
    $$
    then there exists a regular Bohr set $B \leq C''$ such that
    $$
       \langle \mu_{B} \ast \mu_{M_1} \circ \mu_{M_2}, 1_S \rangle \geq 1 - \eps
    $$
    where
    \[
        \rk(B) \leq r + O_{\eps}\left(d^4p^4\right) \quad\text{and}\quad \mu(B) \geq \exp\left( -O_{\eps}(r p\left(d + \log r \right) + d^5 p^5)\right) \mu(C'').
    \]
\end{corollary}

\begin{proof}
    Since $\mu_{M_1} \circ \mu_{M_2}$ is supported on $C' + C'' - x$, we can assume $S$ is supported on the same set. We can also shift $M_2$, $S$ by $x$ so that $M_2 \subseteq C''$ without changing any assumptions or conclusions. Thus, in order to apply \Cref{thm:almostperiod}, it suffices to check that $|C' + C''| \leq 2|C'|$. By regularity and our choice of $\sigma$, we have
    $$
        |C' + C''| = |C' + C'_{\sigma}| \leq (1 + O(\sigma \cdot r))|C'| \leq 2|C'|.
    $$
    The proof concludes by applying \Cref{thm:almostperiod} with $B \coloneqq C'$, $B' \coloneqq C''$ and $M_1,M_2,S$ as given. 
\end{proof}

\subsubsection{Combining steps}
All that remains to prove \Cref{thm:strucvspseudo} (restated below) is to combine the prior steps.
\strVSpsdAbelian*

\begin{proof}
    If we are not in the first condition, then \Cref{lem:nonuniform} implies there is some integer $p = O(d)$ such that 
    $$
        \left \|\frac{1}{\alpha} \sum_{g \in C} \left(1_{A_g} - \alpha_g \mu_R \right) \circ \left(1_{A_g+g} - \alpha_g \mu_{R+g}\right) \right \|_{p(\mu_{C})} \geq \frac{1}{2} \eps \mu(R)^{-1}.
    $$
    We then apply \Cref{lem:unbalancing_int} to infer that 
    $$
        \left \| \frac{1}{\alpha} \sum_{g \in C} 1_{A_g} \circ 1_{A_g} \right \|_{p'(\mu_{C'} \ast \mu_{C'' - x})}\geq \left(1 + \frac{\eps}{16}\right) \mu(R)^{-1},
    $$
    for some shift $x \in G$, where $C' \subseteq C_{\tau_1}$ and $C'' \subseteq C'_{\tau_2}$ are regular Bohr sets for some $\tau_1, \tau_2= O(r^{-1})$ chosen so that $C', C''$ are regular. In particular, the choice of dilates is so that we can later use \Cref{cor:density_increment}. 
    Let $$
        S = \left\{x \in G : \sum_{g \in C} 1_{A_g} \circ 1_{A_g}(x) > (1-\eps/32) \left \| \sum_{g \in C} 1_{A_g} \circ 1_{A_g} \right \|_{p'(\mu_{C'} \ast \mu_{C'' - x})} \right\}.
    $$
    Via \Cref{cor:robust_witness_int}, we obtain $M_1 \subseteq C'$, $M_2 \subseteq C''$ such that 
    $$
        \langle \mu_{M_1} \circ \mu_{M_2}, 1_{S} \rangle \geq 1 - \eps/128,
    $$ 
    where $M_1$ and $M_2$ satisfy the density constraints
    $$
        \min\left(\mu_{B}(M_1), \mu_{B'}(M_2) \right) \geq 2^{-O_{\eps}(dp)}.
    $$
    We then apply \Cref{cor:density_increment} to produce a regular Bohr set $B \leq C''$ 
    such that
    $$
       \langle \mu_{B} \ast \mu_{M_1} \circ \mu_{M_2}, 1_S \rangle \geq 1 - \eps/64,
    $$
    where
    \[
        \rk(B) = r + O_{\eps}(d^8) \quad\text{and}\quad \mu(B) \geq \exp\left(-O_{\eps}(rd^2 + rd \log r + d^{10}) \right) \mu(C'').
    \]
    Since $C = R_{\sigma}$, $C' = R_{\sigma \tau_1}$, and $C'' \subseteq R_{\sigma \tau_1 \tau_2}$, \Cref{lem:bohrsize} implies
    $$
        \mu(C'') \geq \left(\frac{\sigma \tau_1 \tau_2}{4} \right)^{r}\mu(R) \geq \exp\left(-O( r \log(1/\sigma) + r \log r) \right) \mu(R).
    $$
    Combining the above inequalities gives
    $$
        \mu(B) \geq \text{exp}\left(-O_{\epsilon}(r \log(1/\sigma) + rd^2 + rd \log r + d^{10})\right).
    $$ 
    The last step is to note that
    \begin{align*}
        (1 + \eps/80)\mu(R)^{-1} &\leq (1 - \eps/64)\left(1-\eps/32\right) 
        (1+\eps/16)\mu(R)^{-1} \\
        &\leq \langle \mu_B \ast \mu_{M_1} \circ \mu_{M_2}, \mu_{F'} \rangle \\
        &\leq \| \mu_{F'} \ast \mu_B\|_{\infty} \|\mu_{M_1} \circ \mu_{M_2}\|_1 \\
        &= \|\mu_{F'} \ast \mu_B\|_{\infty}. \tag*{\qedhere}
    \end{align*}
\end{proof}


\section*{Acknowledgments} 
We thank an anonymous reviewer for a number of helpful comments. MJ and AO thank Lior Lovett for his constant encouragement.

\bibliographystyle{amsplain}


\begin{dajauthors}
\begin{authorinfo}[michael]
  Michael Jaber\\
  Department of Computer Science, UT Austin\\
  USA\\
  mjjaber\imageat{}cs\imagedot{}utexas\imagedot{}edu \\
  \url{https://michaeljaber.github.io/}
\end{authorinfo}
\begin{authorinfo}[shachar]
  Shachar Lovett\\
  Department of Computer Science and Engineering, UC San Diego\\
  USA\\ 
  slovett\imageat{}ucsd\imagedot{}edu \\
  \url{https://cseweb.ucsd.edu/~slovett/}
\end{authorinfo}
\begin{authorinfo}[anthony]
  Anthony Ostuni\\
  Department of Computer Science and Engineering, UC San Diego\\
  USA\\
  aostuni\imageat{}ucsd\imagedot{}edu \\
  \url{https://aostuni.github.io/}
\end{authorinfo}
\end{dajauthors}

\end{document}

%% file: macro.tex
\newtheorem{theorem}{Theorem}[section]
\newtheorem*{theorem*}{Theorem}
\newtheorem{corollary}[theorem]{Corollary}

\newtheorem{lemma}[theorem]{Lemma}
\newtheorem*{lemma*}{Lemma}
\newtheorem{definition}[theorem]{Definition}

\newtheorem{question}[theorem]{Question}

\newtheorem{observation}[theorem]{Observation}

\newcommand{\R}{\mathbb{R}}
\newcommand{\Z}{\mathbb{Z}}
\newcommand{\F}{\mathbb{F}}

\DeclareMathOperator{\E}{\mathbb{E}}

\newcommand{\ceil}[1]{\lceil {#1} \rceil}
\newcommand{\eps}{\varepsilon}
\newcommand{\Rplus}{\mathbb{R}_{\ge 0}}
\newcommand{\rk}{\mathrm{rk}}